\documentclass[a4paper,twoside,10pt]{article}

\usepackage[english]{babel}
\usepackage[latin1]{inputenc}
\usepackage{lmodern}
\usepackage[T1]{fontenc}
\usepackage{indentfirst}
\usepackage{authblk}
\usepackage[a4paper,top=3cm,bottom=3cm,left=3cm,right=3cm,%
bindingoffset=5mm]{geometry}
\usepackage{lmodern}
\usepackage[T1]{fontenc}
\usepackage{lettrine}
\usepackage{booktabs}

\usepackage{epstopdf}
\usepackage{emp}
\usepackage{amsthm}
\usepackage{amsmath,amssymb,amsthm}
\usepackage{braket}
\usepackage{chngpage}
\usepackage{cases}
\usepackage{graphicx}
\usepackage{graphics}
\usepackage{tabularx}
\usepackage{multirow}
\usepackage{color}

\newcommand{\numberset}{\mathbb}
\newcommand{\N}{\numberset{N}}
\newcommand{\B}{\numberset{B}}
\newcommand{\Pk}{\numberset{P}}
\newcommand{\R}{\numberset{R}}

\renewcommand{\epsilon}{\varepsilon}
\renewcommand{\theta}{\vartheta}
\renewcommand{\rho}{\varrho}
\renewcommand{\phi}{\varphi}
{\left\lbrace\begin{array}{@{}l@{}}}%
{\end{array}\right.}
\theoremstyle{definition}

\theoremstyle{remark}
\newtheorem{test}{Test}[section]
\theoremstyle{remark}
\newtheorem{osservazione}{Remark}[section]

\newcommand{\rr}{\color{red}}

\theoremstyle{plain}
\newtheorem{teorema}{Theorem}[section]
\newtheorem{proposizione}{Proposition}[section]

\newtheorem{lemma}{Lemma}[section]

\usepackage{setspace}

\newcommand\blfootnote[1]{%
  \begingroup
  \renewcommand\thefootnote{}\footnote{#1}%
  \addtocounter{footnote}{-1}%
  \endgroup
}

\usepackage{lipsum}

\author{L. Beir\~ao da Veiga 
\thanks{Dipartimento di Matematica,  Universit\`a degli Studi di Milano, Via Cesare Saldini 50 - 20133 Milano, email: lourenco.beirao@unimi.it} , 
\quad C. {Lovadina  
\thanks{Dipartimento di Matematica,  Universit\`a degli Studi di Pavia, Via Ferrata,1 - 27100 Pavia, email: carlo.lovadina@unipv.it } \quad and \,
G. Vacca  
\thanks{Dipartimento di Matematica,  Universit\`a degli Studi di Bari, Via Edoardo Orabona, 4 - 70125 Bari, email: giuseppe.vacca@uniba.it}}} 

\title{\textbf{Divergence free Virtual Elements for the Stokes problem on polygonal meshes}}
\date{}
\begin{document}
\maketitle

\begin{abstract}
In the present paper we develop a new family of Virtual Elements for the Stokes problem on polygonal meshes. 
By a proper choice of the Virtual space of velocities and the associated degrees of freedom, we can guarantee that the 
final discrete velocity is {\em pointwise} divergence-free, and not only in a relaxed (projected) sense, 
as it happens for more standard elements.
Moreover, we show that the discrete problem is immediately equivalent 
to a reduced problem with less degrees of freedom, thus yielding a very efficient scheme.
We provide a rigorous error analysis of the method and several numerical tests, 
including a comparison with a different Virtual Element choice.
\end{abstract}

\textbf{keywords:} Virtual element method, Polygonal meshes, Stokes Problem, Divergence free approximation

\blfootnote{Work performed while Giuseppe Vacca was visiting the Department of Mathematics of 
the University of Pavia, the support of which is gratefully acknowledged.}

%
%
%

\section{Introduction}
\label{sec:1}

The last decade has seen an increased interest
in developing numerical methods that 
can make use of general polygonal and polyhedral meshes, as opposed to more standard triangular/quadrilateral (tetrahedral/hexahedral) 
grids. Indeed, making use of polygonal meshes brings forth a range of advantages, including for instance automatic use of nonconforming grids, more efficient approximation of geometric data features, better domain meshing 
capabilities, more efficient and easier 
adaptivity, more robustness to mesh deformation, and others. This interest in the literature is also reflected in commercial codes, 
such as CD-Adapco, that have recently included polytopal meshes.

We refer to the recent papers and monographs 
\cite{%
BLS05bis,BLM11,BBL09,BLM11book,Bishop13,JA12,LMSXX,MS11,NBM09,RW12,POLY37,%
ST04,TPPM10,VW13,Wachspress11,DiPietro-Ern-1,Gillette-1,PolyDG-1} %
as a brief representative sample of the increasing list of technologies that make use of polygonal/polyhedral meshes. We mention 
here in particular the polygonal finite elements, that generalize finite elements to polygons/polyhedrons by making use of 
generalized non-polynomial shape functions, and the mimetic discretisation schemes, that combine ideas from the finite difference and 
finite element methods.

The Virtual Element Method (in short, VEM) has been recently introduced in \cite{volley} as a generalization of the finite element method to arbitrary element-geometry. 
The principal idea behind VEM is to use approximated discrete bilinear forms that require only integration of polynomials on the (polytopal) 
element in order to be computed. The resulting discrete solution is conforming and the accuracy granted by such discrete bilinear forms 
turns out to be sufficient to achieve the correct order of convergence. 
Following this approach, VEM is able to make use of very general polygonal/polyhedral meshes without the need to integrate 
complex non-polynomial functions on the elements and without loss of accuracy. Moreover, 
VEM is not restricted to low order converge and can be easily applied to three dimensions and use non convex (even non simply connected) 
elements.
The Virtual Element Method has been applied successfully for a large range of problems, for instance a non
exhaustive list being 
\cite{volley,Brezzi:Marini:plates,VEM-elasticity,projectors,BM13,BFMXX,VEM-stream,GTP14,BBPSXX,mora2015virtual, lovadina, vaccabis, giani}. 
A helpful paper for the implementation of the method is \cite{VEM-hitchhikers}.

The focus of the present paper is on developing a Virtual Element Method for the Stokes problem. In \cite{VEM-elasticity} 
the authors presented a family of Virtual Elements for the linear elasticity problem that are locking free in the incompressible limit. 
As a consequence, the scheme in \cite{VEM-elasticity} can be immediately extended to the Stokes problem, thus yielding a stable VEM family 
(that would be comparable to the Crousiex-Raviart finite element family). 

In the present paper, we develop instead a more efficient and, potentially, accurate method by exploiting in a new way the flexibility of 
the Virtual Element construction. Indeed, we define a new Virtual Element space of velocities such that the associated discrete kernel is 
also pointwise divergence-free. As a consequence, the final velocity discrete solution will have a true vanishing divergence, not only in a 
relaxed (projected) sense as it happens for standard finite elements. As we show in the numerical tests section, this seems to yield a better 
accuracy when compared to the Stokes extension of \cite{VEM-elasticity}; moreover, the divergence-free property is useful when more complex 
problems, such as Navier-Stokes, are considered.

In addition to the above feature, the proposed method carries an additional important advantage. By selecting suitable degrees of 
freedom (DoFs in the sequel), we obtain an automatic orthogonality condition among many pressure DoFs and the associated DoFs for the velocities. 
As a consequence, a large amount of degrees of freedom can be automatically eliminated from the system and one obtains a new reduced problem with 
less degrees of freedom: only one pressure DoF per element and very few internal-to-elements DoFs for the velocities. 

We finally note that the proposed problem is new also on triangles and quadrilaterals, 
allowing for new divergence-free (Virtual) elements with fewer degrees of freedom. 

In brief, the proposed family of Virtual Elements has three advantages: 1) it can be applied to general polygonal meshes, 
2) it yields an exactly divergence-free velocity, 3) it is efficient in terms of number of degrees of freedom. 
In the current work, after developing the method, we prove its stability and convergence properties. Finally, 
we test the method on some benchmark problems and compare it with the Stokes extension of \cite{VEM-elasticity}.

The paper is organized as follows. In Section \ref{sec:2} we introduce the model continuous Stokes problem. 
In Section \ref{sec:3} we present its VEM discretisation. In Section \ref{sec:4} we detail the theoretical features and the convergence analysis 
of the problem. In Section \ref{sec:5} we describe the reduced problem and its properties. 
Finally, in Section \ref{sec:6} we show the numerical tests.


\section{The continuous problem}
\label{sec:2}

We consider the Stokes Problem on a polygon $\Omega \subseteq \R^2$ with homogeneous Dirichlet boundary conditions:
\begin{equation}
\label{eq:stokes primale}
\left\{
\begin{aligned}
& \mbox{ find $(\mathbf{u},p)$ such that}& &\\
& -  \nu \, \boldsymbol{\Delta}    \mathbf{u}  -  \nabla p = \mathbf{f} \qquad  & &\text{in $\Omega$,} \\
& {\rm div} \, \mathbf{u} = 0 \qquad & &\text{in $\Omega$,} \\
& \mathbf{u} = 0  \qquad & &\text{on $\Gamma = \partial \Omega$,}
\end{aligned}
\right.
\end{equation}
where $\mathbf{u}$ and $p$ are the velocity and the pressure fields, respectively. 
Furthermore, $\boldsymbol{\Delta} $, ${\rm div}$, $\boldsymbol{\nabla}$, and $\nabla$ denote the vector Laplacian,  
the divergence, the gradient operator for vector fields and the gradient operator for scalar functions. Finally, $\mathbf{f}$ represents 
the external force, while $\nu$ is the viscosity.

Let us consider the spaces
\begin{equation}
\label{eq:spazi continui}
\mathbf{V}:= \left[ H_0^1(\Omega) \right]^2, \qquad Q:= L^2_0(\Omega) = \left\{ q \in L^2(\Omega) \quad \text{s.t.} \quad \int_{\Omega} q \,{\rm d}\Omega = 0 \right\} 
\end{equation}
with norms
\begin{equation}
\label{eq:norme continue}
\| \mathbf{v}\|_{1} := \| \mathbf{v}\|_{\left[ H^1(\Omega) \right]^2} \quad , \qquad 
\|q\|_Q := \| q\|_{L^2(\Omega)}. 
\end{equation}

We assume $\mathbf{f}\in[H^{-1}(\Omega)]^2$, and $\nu\in L^{0, \infty}(\Omega)$ uniformly positive in $\Omega$.
Let the bilinear forms $a(\cdot, \cdot) \colon \mathbf{V} \times \mathbf{V} \to \R$ and $b(\cdot, \cdot) \colon \mathbf{V} \times Q \to \R$  be defined by:
\begin{equation}
\label{eq:forma a}
a (\mathbf{u},  \mathbf{v}) := \int_{\Omega} \nu \, \boldsymbol{\nabla}   \mathbf{u} : \boldsymbol{\nabla}  \mathbf{v} \,{\rm d} \Omega, \qquad \text{for all $\mathbf{u},  \mathbf{v} \in \mathbf{V}$}
\end{equation}
\begin{equation}
\label{eq:forma b}
b(\mathbf{v}, q) :=  \int_{\Omega} {\rm div} \,\mathbf{v}\, q \,{\rm d}\Omega \qquad \text{for all $\mathbf{v} \in \mathbf{V}$, $q \in Q$.}
\end{equation}
Then a standard variational formulation of Problem \eqref{eq:stokes primale} is:
\begin{equation}
\label{eq:stokes variazionale}
\left\{
\begin{aligned}
& \text{find $(\mathbf{u}, p) \in \mathbf{V} \times Q$, such that} \\
& a(\mathbf{u}, \mathbf{v}) + b(\mathbf{v}, p) = (\mathbf{f}, \mathbf{v}) \qquad & \text{for all $\mathbf{v} \in \mathbf{V}$,} \\
&  b(\mathbf{u}, q) = 0 \qquad & \text{for all $q \in Q$,}
\end{aligned}
\right.
\end{equation}
where
\[
 (\mathbf{f}, \mathbf{v}) := \int_{\Omega} \mathbf{f} \cdot \mathbf{v} \, {\rm d} \Omega .
\]
It is well known that (see for istance \cite{BoffiBrezziFortin}):
\begin{itemize}
\item $a(\cdot, \cdot)$ and $b(\cdot, \cdot)$ are continuous, i.e.
\[
|a(\mathbf{u}, \mathbf{v})| \leq \|a\| \|\mathbf{u}\|_{1}\|\mathbf{v}\|_{1} \qquad \text{for all $\mathbf{u}, \mathbf{v} \in \mathbf{V}$,}
\] 
\[
|b(\mathbf{v}, q)| \leq \|b\| \|\mathbf{v}\|_{1} \|q\|_Q \qquad \text{for all $\mathbf{v} \in \mathbf{V}$ and $q \in Q$;} 
\]
\item $a(\cdot, \cdot)$ is coercive, i.e. there exists a positive constant $\alpha$ such that 
\[
a(\mathbf{v}, \mathbf{v}) \geq \alpha \|\mathbf{v}\|^2_{1} \qquad \text{for all $\mathbf{v} \in \mathbf{V}$;}
\]
\item  the bilinear form $b(\cdot,\cdot) $ satisfies the inf-sup condition, i.e.   
\begin{equation}
\label{eq:inf-sup}
\exists \, \beta >0 \quad \text{such that} \quad \sup_{\mathbf{v} \in \mathbf{V} \, \mathbf{v} \neq \mathbf{0}} \frac{b(\mathbf{u}, q)}{ \|\mathbf{v}\|_{1}} \geq \beta \|q\|_Q \qquad \text{for all $q \in Q$.}
\end{equation}
\end{itemize}
Therefore, Problem \eqref{eq:stokes variazionale} has a unique solution $(\mathbf{u}, p) \in \mathbf{V} \times Q$ such that
\[
 \|\mathbf{u}\|_{1} + \|p\|_Q \leq C \, \|\mathbf{f}\|_{H^{-1}(\Omega)}
\]
with the constant $C$ depending only on $\Omega$.


\section{Virtual formulation for the Stokes problem}
\label{sec:3}

\subsection{Decomposition and virtual element spaces}
\label{sub:3.1}
 
We outline the Virtual Element discretization of Problem \eqref{eq:stokes variazionale}. 
Here and in the rest of the paper the symbol $C$ will indicate a generic positive constant independent of the mesh size that may change at 
each occurrence. Moreover, given any subset $\omega$ in ${\mathbb R}^2$ and $k \in {\mathbb N}$, we will denote by $\Pk_k(\omega)$ the 
polynomials of total degree at most $k$ defined on $\omega$, with the extended notation $\Pk_{-1}(\omega)=\emptyset$.
Let $\set{\mathcal{T}_h}_h$ be a sequence of decompositions of $\Omega$ into general polygonal elements $K$ with
\[
 h_K := {\rm diameter}(K) , \quad
h := \sup_{K \in \mathcal{T}_h} h_K .
\]
We suppose that for all $h$, each element $K$ in $\mathcal{T}_h$ fulfils the following assumptions:
\begin{itemize}
\item $\mathbf{(A1)}$ $K$ is star-shaped with respect to a ball of radius $ \ge\, \gamma \, h_K$, 
\item $\mathbf{(A2)}$ the distance between any two vertexes of $K$ is $\ge c \, h_K$, 
\end{itemize}
where $\gamma$ and $c$ are positive constants. We remark that the hypotheses above, though not too restrictive in many practical cases, 
can be further relaxed, as noted in ~\cite{volley}. 

We also assume that the scalar field $\nu$ is piecewise constant with respect to the decomposition $\mathcal{T}_h$, i.e. $\nu$ is constant on each polygon $K \in \mathcal{T}_h$.

The bilinear forms $a(\cdot,\cdot)$ and $b(\cdot, \cdot)$, the norms $||\cdot||_{1}$ and $||\cdot ||_Q$, can be decomposed into local contributions. Indeed, using obvious notations, we have 
\begin{equation}
\label{eq:forma a locali continue}
a (\mathbf{u},  \mathbf{v}) =: \sum_{K \in \mathcal{T}_h} a^K (\mathbf{u},  \mathbf{v}) \qquad \text{for all $\mathbf{u},  \mathbf{v} \in \mathbf{V}$}
\end{equation}
\begin{equation}
\label{eq:forma b locali continue}
b (\mathbf{v},  q) =: \sum_{K \in \mathcal{T}_h} b^K (\mathbf{v},  q) \qquad \text{for all $\mathbf{v} \in \mathbf{V}$ and $q \in Q$,}
\end{equation}
and 
\begin{equation}
\label{eq:norme locali}
\|\mathbf{v}\|_{1} =: \left(\sum_{K \in \mathcal{T}_h} \|\mathbf{v}\|^2_{1, K}\right)^{1/2} \quad \text{for all $\mathbf{v} \in \mathbf{V}$,} \qquad \|q\|_Q =: \left(\sum_{K \in \mathcal{T}_h} \|q\|^2_{Q, K}\right)^{1/2} \quad \text{for all $q \in Q$.}
\end{equation}

For $k \in \N$, let us define the spaces 
\begin{itemize}
\item $\Pk_k(K)$ the set of polynomials on $K$ of degree $\le k$,
\item $\B_k(K) := \{v \in C^0(\partial K) \quad \text{s.t} \quad v_{|e} \in \Pk_k(e) \quad \forall\mbox{ edge } e \subset \partial K\}$,
\item $\mathcal{G}_{k}(K):= \nabla(\Pk_{k+1}(K)) \subseteq [\Pk_{k}(K)]^2$,
\item $\mathcal{G}_{k}(K)^{\perp} \subseteq [\Pk_{k}(K)]^2$ the $L^2$-orthogonal complement to $\mathcal{G}_{k}(K)$. 
%
\end{itemize}

On each element $K \in \mathcal{T}_h$ we define,  for $k \ge 2$, the following finite dimensional local virtual spaces
\begin{multline}
\label{eq:V_h^K}
\mathbf{V}_h^K := \biggl\{  
\mathbf{v} \in [H^1(K)]^2 \quad \text{s.t} \quad \mathbf{v}_{|{\partial K}} \in [\B_k(\partial K)]^2 \, , \biggr.
\\
\left.
\biggl\{
\begin{aligned}
& - \nu \, \boldsymbol{\Delta}    \mathbf{v}  -  \nabla s \in \mathcal{G}_{k-2}(K)^{\perp},  \\
& {\rm div} \, \mathbf{v} \in \Pk_{k-1}(K),
\end{aligned}
\biggr. \qquad \text{ for some $s \in L^2(K)$}
\quad \right\}
\end{multline}
and
\begin{equation}
\label{eq:Q_h^K}
Q_h^K := \Pk_{k-1}(K). 
\end{equation}
We note that all the operators and equations above are to be intended in the weak sense. In particular, 
the definition of $\mathbf{V}_h^K$ above is associated to a Stokes-like variational problem on $K$.

It is easy to observe that
$
[\Pk_k(K)]^2 \subseteq \mathbf{V}_h^K$,  
and it holds 
\begin{equation}
\label{eq:dimensioni}
 \dim\left([\B_k(\partial K)]^2\right) = 2n_K k, \quad \dim\left(\mathcal{G}_{k-2}(K)^{\perp}\right) = \frac{(k-1)(k-2)}{2} 
\end{equation}
where $n_K$ is the number of edges of the polygon $K$.

It is well-known (see for instance \cite{fortin1991mixed}) that given
\begin{itemize}
\item a polynomial function $\mathbf{g}_b \in [\B_k({\partial K})]^2$,
\item a polynomial function $\mathbf{h} \in \mathcal{G}_{k-2}(K)^{\perp}$,
\item a polynomial function $g \in \Pk_{k-1}(K)$  satisfying the compatibility condition
\begin{equation}\label{compcond}
\int_{K} g \, {\rm d}\Omega = \int_{\partial K} \mathbf{g}_b \cdot \mathbf{n} \, {\rm d}s, 
\end{equation}
\end{itemize}
there exists a unique couple $(\mathbf{v}, s)\in  \mathbf{V}_h^K\times  L^2(K) / \R$ such that
\begin{equation}\label{datasol}
\mathbf{v}_{|\partial K} = \mathbf{g}_b, \quad {\rm div} \, \mathbf{v} = g, \quad 
 - \nu \, \boldsymbol{\Delta}    \mathbf{v}  -  \nabla s = \mathbf{h}.
\end{equation}

Moreover, let us assume that there exist two different data sets
\[
(\mathbf{g}_{b}, \, \mathbf{h}, \, g) \quad \text{and} \quad  (\mathbf{c}_{b}, \, \mathbf{d}, \, c) \in [\B_k({\partial K})]^2 \times 
\mathcal{G}_{k-2}(K)^{\perp} \times \Pk_{k-1}(K), 
\]
both satisfying the compatibility conditions, which correspond respectively to the couples 
$(\mathbf{v}, \,s), (\mathbf{v}, \, t) \in \mathbf{V}_h^K 
\times L^2(K)$ (i.e. same velocity and different pressures). Then it is straightforward to see that 
\[ 
\mathbf{g}_{b} = \mathbf{c}_{b}, \qquad  g = c  \qquad \text{and} \qquad \nabla (s - t) =   \mathbf{d} - \mathbf{h}.
\]
Therefore, we get ${\rm rot} (\mathbf{d} - \mathbf{h}) = 0$, where ${\rm rot}$ is the rotational operator in 2D, i.e. the rotated divergence. 
Since ${\rm rot} \colon \mathcal{G}_{k-2}(K)^{\perp} \to \Pk_{k-3}(K)$ is an isomorphism (see \cite{conforming}), 
we conclude that $\mathbf{d} = \mathbf{h}$. Thus, there is an {\em injective} map $(\mathbf{g}_{b}, \, \mathbf{h}, \, g)\to \mathbf{v}$ 
that associates a given compatible data set   
$(\mathbf{g}_{b}, \, \mathbf{h}, \, g)$ to the velocity field $\mathbf{v}$ that solves~\eqref{datasol}.
It follows that the dimension of $\mathbf{V}_h^K$ is
\begin{equation}
\label{eq:dimensione V_h^K}
\begin{split}
\dim\left( \mathbf{V}_h^K \right) &= \dim\left([\B_k(\partial K)]^2\right) + \dim\left(\mathcal{G}_{k-2}(K)^{\perp}\right) + \left( \dim(\Pk_{k-1}(K)) - 1\right) \\
&= 2n_K k + \frac{(k-1)(k-2)}{2}  + \frac{(k+1)k}{2} - 1.
\end{split}
\end{equation}
For the local space $Q_h^K$  we have
\begin{equation}
\label{eq:dimensione Q_h^K}
\dim(Q_h^K) = \dim(\Pk_{k-1}(K))  = \frac{(k+1)k}{2}.
\end{equation}

We are now ready to introduce suitable sets of degrees of freedom for the local approximation fields.

Given a function $\mathbf{v} \in \mathbf{V}_h^K$ we take the following linear operators $\mathbf{D_V}$, split into four subsets:
\begin{itemize}
\item $\mathbf{D_V1}$:  the values of $\mathbf{v}$ at the vertices of the polygon $K$,
\item $\mathbf{D_V2}$: the values of $\mathbf{v}$ at $k-1$ distinct points of every edge $e \in \partial K$ (for example we can take the $k-1$ internal points of the $(k+1)$-Gauss-Lobatto quadrature rule  in $e$, as suggested in \cite{VEM-hitchhikers}),
\item $\mathbf{D_V3}$: the moments
\[
\int_K \mathbf{v} \cdot \mathbf{g}_{k-2}^{\perp} \, {\rm d}K \qquad \text{for all $\mathbf{g}_{k-2}^{\perp} \in \mathcal{G}_{k-2}(K)^{\perp}$,}
\]
\item $\mathbf{D_V4}$: the moments up to order $k-1$ and greater than zero of ${\rm div} \,\mathbf{v}$ in $K$, i.e.
\[
\int_K ({\rm div} \,\mathbf{v}) \, q_{k-1} \, {\rm d}K \qquad \text{for all $q_{k-1} \in \Pk_{k-1}(K) / \R$.}
\] 
\end{itemize}

Furthermore, for the local pressure, given  $q\in Q_h^K$, we consider the linear operators $\mathbf{D_Q}$:
\begin{itemize}
\item $\mathbf{D_Q}$: the moments up to order $k-1$ of $q$, i.e.
\[
\int_K q \, p_{k-1} \, {\rm d}K \qquad \text{for all $p_{k-1} \in \Pk_{k-1}(K)$.}
\]
\end{itemize}

Since it is obvious that $\mathbf{D_Q}$ is unisolvent with respect to $Q_h^K$, it only remains to prove the unisolvence of $\mathbf{D_V}$. 
We first prove the following Lemma; we recall that all the differential operators are to be intended in the weak sense.

\begin{lemma}
\label{lemma1}
Let $\mathbf{v} \in \mathbf{V}_h^K$ such that $\mathbf{D_V1}(\mathbf{v})=\mathbf{D_V2}(\mathbf{v})=\mathbf{D_V4}(\mathbf{v}) =0$. Then
\begin{equation}\label{eqlemma}
<\nabla \phi , \mathbf{v}>_{K} = 0 \qquad \text{for all $\phi \in L^2(K)$.}
\end{equation}
where, here and in the following, the brackets $<,>_{K}$ denote the duality pair between $H^1_0(K)^2$ and its dual $H^{-1}(K)^2$. 
\end{lemma}

\begin{proof}
It is clear that $\mathbf{D_V1}(\mathbf{v})=\mathbf{D_V2}(\mathbf{v}) =0$ implies $\mathbf{v}_{| \partial K} \equiv \mathbf{0}$. 
Therefore $\mathbf{v} \in H^1_0(K)$ and it holds
$$
<\nabla \phi , \mathbf{v}>_{K} = - \int_K ({\rm div} \, \mathbf{v}) \phi \, {\rm d}K .
$$
Now, since $\mathbf{v} \in \mathbf{V}_h^K$, there exists $p_{k-1} \in \Pk_{k-1}(K)$ such that  ${\rm div} \,\mathbf{v} = p_{k-1}$.
Furthermore, by the divergence Theorem, we infer that $p_{k-1} \in \Pk_{k-1}(K)/\R$. Since $\mathbf{D_V4}(\mathbf{v})=0$, we get:
\[
\int_K ({\rm div} \, \mathbf{v})^2 \, {\rm d}K = 
\int_K {\rm div} \, \mathbf{v} \, p_{k-1} \, {\rm d}K = 0 .
\]
Therefore, ${\rm div} \, \mathbf{v}=0$ and~\eqref{eqlemma} follows. 
\end{proof}
We now prove the following result. 
\begin{proposizione} 
\label{thm1}
The linear operators $\mathbf{D_V}$ are a unisolvent set of degrees of freedom for the virtual space $\mathbf{V}_h^K$.
\end{proposizione}

\begin{proof}
We start noting that the dimension of $\mathbf{V}_h^K$ equals the number of functionals in $\mathbf{D_V}$ and thus we only need to show that 
if all the values $\mathbf{D_V}(\mathbf{v})$ vanish for a given $\mathbf{v} \in \mathbf{V}_h^K$, then $\mathbf{v}=0$.
Since $\mathbf{D_V1}(\mathbf{v})=\mathbf{D_V2}(\mathbf{v})=0$ implies $\mathbf{v} \equiv  \mathbf{0}$ on $\partial K$, 
we have $\mathbf{v} \in H^1_0(K)$. Therefore
\[
\int_K \nu \, \boldsymbol{\nabla}   \mathbf{v} : \boldsymbol{\nabla}  \mathbf{v} \,{\rm d}K = 
- \nu < \boldsymbol{\Delta}\mathbf{v}  , \mathbf{v} >_{K}.
\]

Moreover, since $\mathbf{v} \in \mathbf{V}_h^K$, there exists a scalar function 
$s \in  L^2(K)$ and $\mathbf{g}_{k-2}^{\perp} \in \mathcal{G}_{k-2}(K)^{\perp}$, such that
\[
\nu \, \boldsymbol{\Delta}    \mathbf{v} = - \nabla s - \mathbf{g}_{k-2}^{\perp} \quad \mbox{ in $H^{-1}(K)^2$}.
\]
Then
\begin{equation}\label{L:1}
\int_K \nu \, \boldsymbol{\nabla}   \mathbf{v} : \boldsymbol{\nabla}  \mathbf{v} \,{\rm d}K = 
<\nabla s , \mathbf{v} >_K 
+ \int_K \mathbf{g}_{k-2}^{\perp}\cdot \mathbf{v} \, {\rm d}K.
\end{equation}
The first term at the right-hand side is zero from Lemma~\ref{lemma1}, while the second term vanishes because of 
the assumption $\mathbf{D_V3}(\mathbf{v})=0$. Then $\mathbf{D_V}(\mathbf{v})=0$ implies $\mathbf{v}=0$, and the proof is complete.
\end{proof}

We now define the global virtual element spaces as
\begin{equation}
\label{eq:V_h}
\mathbf{V}_h := \{ \mathbf{v} \in [H^1_0(\Omega)]^2 \quad \text{s.t} \quad \mathbf{v}_{|K} \in \mathbf{V}_h^K  \quad \text{for all $K \in \mathcal{T}_h$} \}
\end{equation} 
and
\begin{equation}
\label{eq:Q_h}
Q_h := \{ q \in L_0^2(\Omega) \quad \text{s.t.} \quad q_{|K} \in  Q_h^K \quad \text{for all $K \in \mathcal{T}_h$}\},
\end{equation}
with the obvious associated sets of global degrees of freedom. A simple computation shows that it holds:
\begin{equation}\label{eq:Vdofs}
\dim(\mathbf{V}_h) = n_P \left( \frac{(k+1)k}{2} -1  +  \frac{(k-1)(k-2)}{2} \right) 
+ 2 (n_V + (k-1) n_E) 
\end{equation}
and
\begin{equation}\label{eq:Qdofs}
\dim(Q_h) = n_P \frac{(k+1)k}{2} - 1 ,
\end{equation}
where $n_P$ (resp.,  $n_E$ and $n_V$) is the number of elements (resp., internal edges and vertexes) in $\mathcal{T}_h$.

We also remark that
\begin{equation}\label{eq:divfree}
{\rm div}\, \mathbf{V}_h\subseteq Q_h .
\end{equation}
\begin{osservazione}
The space $\mathcal{G}_{k-2}(K)^{\perp}$ that defines the degrees of freedom $\boldsymbol{D_V3}$ can be replaced by any space $ \mathcal{G}_{k-2}(K)^{\oplus} \subseteq [\Pk_{k-2}(K)]^2$ that satisfies 
\[
[\Pk_{k-2}(K)]^2 = \mathcal{G}_{k-2}(K) \oplus \mathcal{G}_{k-2}(K)^{\oplus}.
\]
An example is given by the space $\mathcal{G}_{k-2}(K)^{\oplus}:= \mathbf{x}^{\perp} [\Pk_{k-3}(K)]^2$ with $\mathbf{x}^{\perp}:= (x_2, -x_1)$.
\end{osservazione}

\begin{osservazione}
\label{oss:3.1}
We have built a new $H^1$-conforming (vector valued) virtual space for the velocity vector field, different from the more standard one presented in \cite{VEM-elasticity} for the elasticity problem. In fact, the classical approach is to consider the local virtual space
\begin{equation}
\label{eq:V_h^Kclassic}
\widetilde{\mathbf{V}}_h^K := \left\{  
\mathbf{v} \in [H^1(K)]^2 \quad \text{s.t} \quad \mathbf{v}_{|{\partial K}} \in [\B_k(\partial K)]^2 \, , \quad
 \nu \, \boldsymbol{\Delta}    \mathbf{v}   \in [\Pk_{k-2}(K)]^2 \right\}
\end{equation}
with local degrees of freedom:
\begin{itemize}
\item $\mathbf{\widetilde{D}_V1}$: the values of $\mathbf{v}$ at each vertex of the polygon $K$,
\item $\mathbf{\widetilde{D}_V2}$: the values of $\mathbf{v}$ at $k-1$ distinct points of every edge $e \in \partial K$,
\item $\mathbf{\widetilde{D}_V3}$: the moments up to order $k-2$, i.e.
\[
\int_K \mathbf{v} \, \cdot  \mathbf{q}_{k-2} \, {\rm d}K \qquad \text{for all $\mathbf{q}_{k-2} \in [\Pk_{k-2}(K)]^2$.}
\] 
\end{itemize} 
It can be easily checked that, for all $k$, the dimension of the spaces \eqref{eq:V_h^K} and \eqref{eq:V_h^Kclassic} are the same. 
On the other hand our local virtual space~\eqref{eq:V_h^K} is, in some sense, designed to solve a Stokes-like Problem element-wise,  
while the virtual space in \eqref{eq:V_h^Kclassic} is designed to solve a classical Laplacian problem. As shown in the following, 
although both spaces can be used, the new choice \eqref{eq:V_h^K} is better for the problem under consideration.
\end{osservazione}


\subsection{The discrete bilinear forms}
\label{sub:3.2}

We now define discrete versions of the bilinear form $a(\cdot, \cdot)$ (cf.~\eqref{eq:forma a}), and of the bilinear form $b(\cdot, \cdot)$ (cf.~\eqref{eq:forma b}). For what concerns $b(\cdot, \cdot)$, we simply  set 

\begin{equation}\label{bhform}
b(\mathbf{v}, q) = \sum_{K \in \mathcal{T}_h} b^K(\mathbf{v}, q) = \sum_{K \in \mathcal{T}_h}  \int_K {\rm div} \, \mathbf{v} \, q \,{\rm d}K \qquad \text{for all $\mathbf{v} \in \mathbf{V}_h$, $q \in Q_h$},
\end{equation}
i.e. we do not introduce any approximation of the bilinear form. We notice that~\eqref{bhform}  
is computable from the degrees of freedom $\mathbf{D_V1}$, $\mathbf{D_V2}$ and $\mathbf{D_V4}$, since $q$ is polynomial in each element $K \in \mathcal{T}_h$. 
The construction of a computable approximation of the bilinear form $a(\cdot, \cdot)$ on the virtual space $\mathbf{V}_h$ is more involved.
First of all, we observe that  $\forall\,\mathbf{q}\in [\Pk_k(K)]^2 $ and $\forall\,\mathbf{v}\in \mathbf{V}_h^K $, the quantity   $a^K(\mathbf{q}, \mathbf{v})$ is exactly computable. Indeed, we have
\begin{equation}\label{eq:aKcomp}
a^K (\mathbf{q},  \mathbf{v}) = \int_{K} \nu \, \boldsymbol{\nabla}   \mathbf{q} : \boldsymbol{\nabla}  \mathbf{v} \,{\rm d}K = 
- \int_{K} \nu \, \boldsymbol{\Delta}   \mathbf{q}  \cdot \mathbf{v} \,{\rm d}K  + \int_{\partial K} (\nu \, \boldsymbol{\nabla}    \mathbf{q} \, \mathbf{n}) \cdot \mathbf{v} \,{\rm d}s.
\end{equation}
Since $ \nu \, \boldsymbol{\Delta}   \mathbf{q}  \in [\Pk_{k-2}(K)]^2$, there exists a unique $q_{k-1} \in \Pk_{k-1}(K) / \R$ and $\mathbf{g}_{k-2}^{\perp} \in \mathcal{G}_{k-2}^{\perp}(K)$, such that 
\begin{equation}\label{eq:dcomp}
\nu \, \boldsymbol{\Delta}   \mathbf{q}= \nabla q_{k-1} + \mathbf{g}_{k-2}^{\perp}.
\end{equation}
Therefore, we get
\begin{equation}\label{eq:aKcomp2}
\begin{split}
a^K (\mathbf{q},  \mathbf{v}) &= - \int_{K} \nabla q_{k-1} \cdot \mathbf{v} \,{\rm d}K    - \int_K  \mathbf{g}_{k-2}^{\perp} \cdot \mathbf{v} \,{\rm d}K + \int_{\partial K} (\nu \, \boldsymbol{\nabla}    \mathbf{q} \, \mathbf{n}) \cdot \mathbf{v} \,{\rm d}s \\
&= \int_{K}  q_{k-1} \, {\rm div} \, \mathbf{v} \,{\rm d}K -  \int_K \mathbf{g}_{k-2}^{\perp} \cdot \mathbf{v} \,{\rm d}K + \int_{\partial K} (\nu \, \boldsymbol{\nabla}    \mathbf{q} \, \mathbf{n} - q_{k-1} \mathbf{n}) \cdot \mathbf{v} \,{\rm d}s .
\end{split}
\end{equation}
The first term in the right-hand side is computable from $\mathbf{D_V4}$, the second term from $\mathbf{D_V3}$ and the boundary term from $\mathbf{D_V1}$ and $\mathbf{D_V2}$. 
However, for an arbitrary pair $(\mathbf{w},\mathbf{v} )\in \mathbf{V}_h^K \times \mathbf{V}_h^K $, the quantity $a_h^K(\mathbf{w}, \mathbf{v})$ is not computable.  
We now define a computable discrete local bilinear form
\begin{equation}
\label{eq:a_h^K} 
a_h^K(\cdot, \cdot) \colon \mathbf{V}_h^K \times \mathbf{V}_h^K \to \R
\end{equation}
approximating the continuous form $a^K(\cdot, \cdot)$, and satisfying the following properties:
\begin{itemize}
\item $\mathbf{k}$\textbf{-consistency}: for all $\mathbf{q} \in [\Pk_k(K)]^2$ and $\mathbf{v}_h \in \mathbf{V}_h^K$
\begin{equation}\label{eq:consist}
a_h^K(\mathbf{q}, \mathbf{v}_h) = a^K( \mathbf{q}, \mathbf{v}_h);
\end{equation}
\item \textbf{stability}:  there exist  two positive constants $\alpha_*$ and $\alpha^*$, independent of $h$ and $K$, such that, for all $\mathbf{v}_h \in \mathbf{V}_h^K$, it holds
\begin{equation}\label{eq:stabk}
\alpha_* a^K(\mathbf{v}_h, \mathbf{v}_h) \leq a_h^K(\mathbf{v}_h, \mathbf{v}_h) \leq \alpha^* a^K(\mathbf{v}_h, \mathbf{v}_h).
\end{equation}
\end{itemize}
For all $K \in \mathcal{T}_h$, we introduce the energy projection ${\Pi}_{k}^{\nabla,K} \colon \mathbf{V}_h^K \to [\Pk_k(K)]^2$, defined by 
\begin{equation}
\label{eq:Pi_k^K}
\left\{
\begin{aligned}
& a^K(\mathbf{q}_k, \mathbf{v}_h - \, {\Pi}_{k}^{\nabla,K}   \mathbf{v}_h) = 0 \qquad  \text{for all $\mathbf{q}_k \in [\Pk_k(K)]^2$,} \\
& P^{0,K}(\mathbf{v}_h - \,  {\Pi}_{k}^{\nabla,K}  \mathbf{v}_h) = \mathbf{0} \, ,
\end{aligned}
\right.
\end{equation} 
where $P^{0,K}$ is the $L^2$-projection operator onto the constant functions defined on $K$. It is immediate to check that the energy projection is well defined.  Moreover, it clearly holds ${\Pi}_{k}^{\nabla,K} \mathbf{q}_k = \mathbf{q}_k$ for all $\mathbf{q}_k \in \Pk_k(K)$.

\begin{osservazione}
Since $a^K(\mathbf{q}_k, \mathbf{v}_h)$ is computable (see \eqref{eq:aKcomp2} and the subsequent discussion), it follows that 
the operator ${\Pi}_{k}^{\nabla,K}$ is computable in terms of the degrees of freedom $\mathbf{D_V}$.
\end{osservazione}

As usual in the VEM framework, we now introduce a (symmetric) stabilizing bilinear form $\mathcal{S}^K \colon \mathbf{V}_h^K \times \mathbf{V}_h^K \to \R$, that satisfies
\begin{equation}
\label{eq:S^K}
c_* a^K(\mathbf{v}_h, \mathbf{v}_h) \leq  \mathcal{S}^K(\mathbf{v}_h, \mathbf{v}_h) \leq c^* a^K(\mathbf{v}_h, \mathbf{v}_h) \qquad \text{for all $\mathbf{v}_h \in \mathbf{V}_h$ such that ${\Pi}_{k}^{\nabla,K} \mathbf{v}_h= \mathbf{0}$}.
\end{equation}
Above,  $c_*$ and $c^*$ are two positive  constants, independent of $h$ and $K$.

Then, we can set
\begin{equation}
\label{eq:a_h^K def}
a_h^K(\mathbf{u}_h, \mathbf{v}_h) := a^K \left({\Pi}_{k}^{\nabla,K} \mathbf{u}_h, {\Pi}_{k}^{\nabla,K} \mathbf{v}_h \right) + \mathcal{S}^K \left((I -{\Pi}_{k}^{\nabla,K}) \mathbf{u}_h, (I -{\Pi}_{k}^{\nabla,K}) \mathbf{v}_h \right)
\end{equation}
for all $\mathbf{u}_h, \mathbf{v}_h \in \mathbf{V}_h^K$.

It is easy to see that Definition~\eqref{eq:Pi_k^K} and estimates~\eqref{eq:S^K} imply the consistency and the stability of the bilinear form $a_h^K(\cdot, \cdot)$.

\begin{osservazione}
Following standard argument of VEM techniques, we essentially require that the stabilizing term $\mathcal{S}^K(\mathbf{v}_h, \mathbf{v}_h)$ scales as $a^K(\mathbf{v}_h, \mathbf{v}_h)$. In particular, under our assumptions on the mesh,  the stabilizing term can be constructed using the tools presented in \cite{volley, VEM-hitchhikers}.
\end{osservazione}

Finally we define the global approximated bilinear form $a_h(\cdot, \cdot) \colon \mathbf{V}_h \times \mathbf{V}_h \to \R$ by simply summing the local contributions:
\begin{equation}
\label{eq:a_h}
a_h(\mathbf{u}_h, \mathbf{v}_h) := \sum_{K \in \mathcal{T}_h}  a_h^K(\mathbf{u}_h, \mathbf{v}_h) \qquad \text{for all $\mathbf{u}_h, \mathbf{v}_h \in \mathbf{V}_h$.}
\end{equation}

\subsection{Load term approximation}
\label{sub:3.3}

The last step consists in constructing a computable approximation of the right-hand side $(\mathbf{f}, \mathbf{v})$ 
in \eqref{eq:stokes variazionale}. Let $K \in \mathcal{T}_h$, and let $\Pi^{0,K}_{k-2} \colon [L^2(K)]^2 \to[\Pk_{k-2}(K)]^2$ be 
the $L^2(K)$ projection operator onto the space $[\Pk_{k-2}(K)]^2$. Then, we define the approximated load term $\mathbf{f}_h$ as 
\begin{equation}
\label{eq:f_h}
\mathbf{f}_h := \Pi_{k-2}^{0,K} \mathbf{f} \qquad \text{for all $K \in \mathcal{T}_h$,}
\end{equation}
and consider:
\begin{equation}
\label{eq:right}
(\mathbf{f}_h, \mathbf{v}_h)  = \sum_{K \in \mathcal{T}_h} \int_K \mathbf{f}_h \cdot \mathbf{v}_h \, {\rm d}K = \sum_{K \in \mathcal{T}_h} \int_K \Pi_{k-2}^{0,K} \mathbf{f} \cdot \mathbf{v}_h \, {\rm d}K = \sum_{K \in \mathcal{T}_h} \int_K \mathbf{f} \cdot \Pi_{k-2}^{0,K}  \mathbf{v}_h \, {\rm d}K.
\end{equation}
We observe that \eqref{eq:right} can be exactly computed for all $\mathbf{v}_h \in \mathbf{V}_h$. In fact,  $\Pi_{k-2}^{0,K}  \mathbf{v}_h$
is computable in terms of the degrees of freedom $\mathbf{D_V}$: for all $\mathbf{q}_{k-2} \in [\Pk_{k-2}(K)]^2$ we have
\[
\int_K  \Pi_{k-2}^{0,K}  \mathbf{v}_h\cdot \mathbf{q}_{k-2} \, {\rm d}K = \int_K    \mathbf{v}_h\cdot \mathbf{q}_{k-2} \, {\rm d}K = 
\int_K    \mathbf{v}_h\cdot \nabla q_{k-1} \, {\rm d}K  + \int_K    \mathbf{v}_h\cdot \mathbf{g}_{k-2}^{\perp} \, {\rm d}K 
\]
for suitable $q_{k-1} \in \Pk_{k-1}(K)$ and $\mathbf{g}_{k-2}^{\perp} \in \mathcal{G}_{k-2}(K)^{\perp}$. As a consequence, we get
\[
\int_K  \Pi_{k-2}^{0, K}  \mathbf{v}_h\cdot \mathbf{q}_{k-2} \, {\rm d}K = -\int_K  {\rm div} \, \mathbf{v}_h \, q_{k-1} \, {\rm d}K + \int_{\partial K} q_{k-1} \mathbf{v}_h \cdot \mathbf{n} \, {\rm d}s + \int_K   \mathbf{v}_h\cdot \mathbf{g}_{k-2}^{\perp} \, {\rm d}K , 
\]
and the right-hand side is directly computable from $\mathbf{D_V}$.

Furthermore, the following result concerning a $H^{-1}$-type norm, can be proved using standard arguments.

\begin{lemma}
\label{lemma2}
Let $\mathbf{f}_h$ be defined as in \eqref{eq:f_h},  and let us assume $\mathbf{f} \in H^{k-1}(\Omega)$. Then, for all $\mathbf{v}_h \in \mathbf{V}_h$, it holds
\[
\left|( \mathbf{f}_h - \mathbf{f}, \mathbf{v}_h ) \right| \leq C h^k |\mathbf{f}|_{k-1} \|\mathbf{v}_h\|_{1}.
\]
\end{lemma}


\subsection{The discrete problem}\label{sec:discrete}

We are now ready to state the proposed discrete problem. Referring to~\eqref{bhform}, \eqref{eq:a_h} and~\eqref{eq:f_h}-\eqref{eq:right}, we consider the virtual element problem:
\begin{equation}
\label{eq:stokes virtual}
\left\{
\begin{aligned}
& \text{find $(\mathbf{u}_h, p_h) \in \mathbf{V}_h \times Q_h$, such that} \\
& a_h(\mathbf{u}_h, \mathbf{v}_h) + b(\mathbf{v}_h, p_h) = (\mathbf{f}_h, \mathbf{v}_h) \qquad & \text{for all $\mathbf{v}_h \in \mathbf{V}_h$,} \\
&  b(\mathbf{u}_h, q_h) = 0 \qquad & \text{for all $q_h \in Q_h$.}
\end{aligned}
\right.
\end{equation}
By construction (see~\eqref{eq:stabk}, \eqref{eq:S^K} and~\eqref{eq:a_h^K def}) the discrete bilinear form $a_h(\cdot,\cdot)$ is (uniformly) 
stable with respect to the $\mathbf{V}$ norm. Therefore, the existence and the uniqueness of the solution to 
Problem~\eqref{eq:stokes virtual} will follow if a suitable inf-sup condition is fulfilled, which is the topic of Section~\ref{sub:4.1}.

We also remark that the second equation of~\eqref{eq:stokes virtual}, along with property~\eqref{eq:divfree}, implies that the discrete 
velocity $\mathbf{u}_h\in \mathbf{V}_h$ is {\em exactly divergence-free}. 
More generally, introducing the kernels:
\begin{equation}
\label{eq:Z}
\mathbf{Z} := \{\mathbf{v} \in \mathbf{V} \quad \text{s.t.}  \quad  b(\mathbf{v}, q) =0 \quad \text{for all $q \in Q$}\}
\end{equation}
and
\begin{equation}
\label{eq:Z_h}
\mathbf{Z}_h := \{ \mathbf{v}_h \in \mathbf{V}_h \quad \text{s.t.} \quad b(\mathbf{v}_h, q_h) = 0 \quad \text{for all $q_h \in Q_h$}\},
\end{equation}
it is immediate to check that
\begin{equation}\label{kernincl}
\mathbf{Z}_h \subseteq \mathbf{Z} .
\end{equation}
%


\section{Theoretical results}
\label{sec:4}

We begin by proving an approximation result for the virtual local space $\mathbf{V}_h$. First of all, let us recall a classical result by Scott-Dupont (see \cite{MR2373954}).

\begin{lemma}
\label{lm:scott}
Let $K \in \mathcal{T}_h$, then for all $\mathbf{u} \in [H^{s+1}(K)]^2$ with $0 \leq s \leq k$, there exists a  polynomial function $\mathbf{u}_{\pi} \in [{\Pk}_k(K)]^2$,  such that
\begin{equation}
\label{eq:scott}
\|\mathbf{u} - \mathbf{u}_{\pi}\|_{0,K} + h_K |\mathbf{u} -\mathbf{u}_{\pi} |_{1,K} \leq C h_K^{s+1}| \mathbf{u}|_{s+1,K}.
\end{equation}
\end{lemma}

We have the following proposition.

\begin{proposizione}
\label{thm:interpolation}
Let $\mathbf{u} \in \mathbf{V} \cap [H^{s+1}(\Omega)]^2$ with $0 \leq s \leq k$. Under the assumption $\mathbf{(A1)}$ and $\mathbf{(A2)}$ 
on the decomposition $\mathcal{T}_h$, there exists $\mathbf{u}_I \in \mathbf{V}_h$  such that
\begin{equation}
\label{eq:interpolation}
\|\mathbf{u} - \mathbf{u}_I\|_{0,K} + h_K| \mathbf{u} - \mathbf{u}_I |_{1,K} \leq C h_K^{s+1}|\mathbf{u}|_{s+1,D(K)}
\end{equation}
where $C$ is a constant independent of $h$, and $D(K)$ denotes the ``diamond'' of $K$, i.e. the union of the 
polygons in $\mathcal{T}_h$ intersecting $K$.
\end{proposizione}

\begin{proof}
The proof follows the guidelines of Proposition~4.2 in \cite{mora2015virtual}. For each polygon $K \in \mathcal{T}_h$, let us consider 
the triangulation $\mathcal{T}_h^K$ of $K$ obtained by joining each vertex of $K$ with the center of the ball with respect to which $K$ is 
star-shaped. Set now $\widehat{\mathcal{T}}_h := \bigcup_{K \in \mathcal{T}_h} \mathcal{T}_h^K$, which  is a 
triangular decomposition of the domain $\Omega$.

Let $\mathbf{u}_c$ be the Cl{\'e}ment interpolant of order $k$ of the function $\mathbf{u}$, relative to the triangular  decomposition  
$\widehat{\mathcal{T}}_h$ (see \cite{clement}). Then $\mathbf{u}_c \in [H^1(\Omega)]^2$ and it holds
\begin{equation}\label{eq:clemest}
\|\mathbf{u} - \mathbf{u}_c\|_{0,K} + h_K | \mathbf{u} - \mathbf{u}_c |_{1,K} \leq C h_K^{s+1}|\mathbf{u}|_{s+1,D(K)}.
\end{equation}
Let, for each  polygon $K$, $\mathbf{u}_{\pi}$ be the polynomial approximation of $\mathbf{u}$ as in Lemma \ref{lm:scott}. Then we have:
\begin{equation}\label{eq:upidec}
\nu \, \Delta  \mathbf{u}_{\pi} = \nabla p_{\pi} + \mathbf{g}_{\pi}^{\perp},
\end{equation}
for suitable $p_{\pi} \in \Pk_{k-1}(K)$ and $ \mathbf{g}_{\pi}^{\perp} \in \mathcal{G}_{k-2}(K)^{\perp}$. 
Let $p_c := \Pi_{k-1}^{0,K}({\rm div}\, \mathbf{u}_c)$ for all $K \in \mathcal{T}_h$. We introduce the following local Stokes problem
\begin{equation}
\label{eq:stokesinterpolant}
\left\{
\begin{aligned}
& - \nu \, \boldsymbol{\Delta}   \mathbf{u}_I -  \nabla s = -\mathbf{g}_{\pi}^{\perp} \qquad  & &\text{in $K$,} \\
& {\rm div} \, \mathbf{u}_I = p_c \qquad & &\text{in $K$,} \\
& \mathbf{u}_I = \mathbf{u}_c  \qquad & &\text{on $\partial K$.}
\end{aligned}
\right.
\end{equation}
It is straightforward to check that $\mathbf{u}_I \in \mathbf{V}_h^K$. Furthermore, since $\mathbf{u}_I = \mathbf{u}_c$ on each boundary 
$\partial K $, $\mathbf{u}_I \in [H^1(\Omega)]^2$. We infer that $\mathbf{u}_I \in \mathbf{V}_h$. We now prove that $\mathbf{u}_I$
satisfies estimate~\eqref{eq:interpolation}.
We consider the following auxiliary local Stokes problem
\begin{equation}\label{eq:interpaux}
\left\{
\begin{aligned}
& - \nu \, \boldsymbol{\Delta}\,  \widetilde{\mathbf{u}} -  \nabla \widetilde{s} = -\mathbf{g}_{\pi}^{\perp} \qquad  & &\text{in $K$,} \\
& {\rm div} \, \widetilde{\mathbf{u}} = {\rm div}\, \mathbf{u}_c \qquad & &\text{in $K$,} \\
& \widetilde{\mathbf{u}} = \mathbf{u}_c  \qquad & &\text{on $\partial K$.}
\end{aligned}
\right.
\end{equation}
By \eqref{eq:interpaux} and~\eqref{eq:upidec}, we get
\begin{equation}
\left\{
\begin{aligned}
& - \nu \, \boldsymbol{\Delta}   (\mathbf{u}_{\pi} -  \widetilde{\mathbf{u}})   -  \nabla (- p_{\pi} - \widetilde{s}) = \mathbf{0} \qquad  & &\text{in $K$,} \\
& {\rm div} \, \left( \mathbf{u}_{\pi} - \widetilde{\mathbf{u}} \right) = {\rm div}\, \left( \mathbf{u}_{\pi} - \mathbf{u}_c \right) \qquad & &\text{in $K$,} \\
& \mathbf{u}_{\pi} - \widetilde{\mathbf{u}} = \mathbf{u}_{\pi} - \mathbf{u}_c  \qquad & &\text{on $\partial K$.}
\end{aligned}
\right.
\end{equation}
Therefore we get
\[
\begin{split}
|\mathbf{u}_{\pi} -  \widetilde{\mathbf{u}}|_{1, K} &= 
\inf \{ |\mathbf{z}|_{1, K} \, \text{: $\mathbf{z} \in [H^1(K)]^2$,  ${\rm div} \, \mathbf{z} ={\rm div}\, \left( \mathbf{u}_{\pi} - 
\mathbf{u}_c \right)$ and  $\mathbf{z} = \mathbf{u}_{\pi} - \mathbf{u}_c$ on $\partial K$} \} .
\end{split}
\]
Choosing $\mathbf{z} = \mathbf{u}_{\pi} -  \mathbf{u}_c $, by Lemma \ref{lm:scott} and estimates \eqref{eq:scott} and \eqref{eq:clemest}, we obtain
\begin{equation}\label{eq:intest0}
|\mathbf{u}_{\pi} -  \widetilde{\mathbf{u}}|_{1, K}  \leq |\mathbf{u}_{\pi} -  \mathbf{u}_c |_{1, K} \leq 
\ |\mathbf{u}_{\pi} -  \mathbf{u} |_{1, K}  + |\mathbf{u} -  \mathbf{u}_c|_{1, K}
\leq C\, h_K^s |\mathbf{u}|_{s+1,D(K)}.
\end{equation} 
Subtracting \eqref{eq:stokesinterpolant} from \eqref{eq:interpaux}, we have
\[
\left\{
\begin{aligned}
& - \nu \, \boldsymbol{\Delta}   (\widetilde{\mathbf{u}} - \mathbf{u}_I)  -  \nabla (\widetilde{s} - s) = \mathbf{0} \qquad  & &\text{in $K$,} \\
& {\rm div} \, \left(\widetilde{\mathbf{u}} - \mathbf{u}_I \right) = {\rm div}\,  \mathbf{u}_c - p_c \qquad & &\text{in $K$,} \\
&  \widetilde{\mathbf{u}}  - \mathbf{u}_I= \mathbf{0} \qquad & &\text{on $\partial K$.}
\end{aligned}
\right.
\]
Using the standard theory of saddle point problems (see for instance \cite{BoffiBrezziFortin}), we get
\[
|\widetilde{\mathbf{u}} - \mathbf{u}_I|_{1, K}  \leq \frac{1}{\beta(K)} \left( 1 + \frac{\|a^K\|}{\alpha^K}  \right) \|{\rm div}\,  \mathbf{u}_c - p_c\|_{0,K}
\]
where $\beta(K)$ is the inf-sup constant on the polygon $K$ (cf \eqref{eq:inf-sup}) and $\|a^K\|$ and $\alpha^K$ denote respectively the norm and the coercivity constant of $a^K(\cdot, \cdot)$. It is straightforward to check that
\[
\|a^K\| = \nu \qquad \text{and} \qquad \alpha^K \geq \frac{\nu}{1 + h_K^2}.
\]
Therefore, recalling that $p_c := \Pi_{k-1}^{0,K}({\rm div}\, \mathbf{u}_c)$, using first the triangle inequality, then estimate \eqref{eq:clemest} 
and standard estimates, we have
\[
\begin{split}
|\widetilde{\mathbf{u}} - \mathbf{u}_I|_{1, K}  &\leq  \frac{2 + h_K^2}{\beta(K)} 
\left( \left \| (I - \Pi_{k-1}^{0,K}) \, ({\rm div}  \,\mathbf{u} - {\rm div} \,\mathbf{u}_c) \right \|_{0,K} +  
\left \| (I - \Pi_{k-1}^{0,K}) \, {\rm div}  \,\mathbf{u}\right \|_{0,K} \right) \\
& \leq \frac{C}{\beta(K)} \, \left(  \| {\rm div} (\mathbf{u} - 
\mathbf{u}_c) \|_{0,K} +  h_K^s|{\rm div} \, \mathbf{u}|_{s,K} \right) \\
& \leq \frac{C}{\beta(K)} \, \left( |\mathbf{u} - \mathbf{u}_c|_{1,K} + h_K^s | \mathbf{u} |_{s+1,K} \right) 
\leq  \frac{C}{\beta(K)} \, h_K^s | \mathbf{u} |_{s+1,D(K)}.
\end{split}
\]
By assumption $\mathbf{(A1)}$ and using the results in \cite{constant,duran}, the inf-sup constant $\beta(K)$ is uniformly 
bounded from below: there exists $c>0$, independent of $h$, such that  $\beta(K) \geq c$ for all $K \in \mathcal{T}_h$. 
Therefore, it holds
\begin{equation}\label{eq:intest1}
|\widetilde{\mathbf{u}} - \mathbf{u}_I|_{1, K}  \leq C \, h_K^s | \mathbf{u}|_{s+1,D(K)}.
\end{equation}
The triangle inequality together with estimates \eqref{eq:scott}, \eqref{eq:intest0} and \eqref{eq:intest1}, give
\begin{equation}\label{eq:locest1}
|\mathbf{u} - \mathbf{u}_I|_{1, K} \le  |\mathbf{u} - \mathbf{u}_{\pi}|_{1,K} + 
|\mathbf{u}_{\pi} - \widetilde{\mathbf{u}}|_{1,K}  + |\widetilde{\mathbf{u}} - \mathbf{u}_I|_{1,K}  \leq  C h_K^s | \mathbf{u} |_{s+1,D(K)}
\end{equation}
Furthermore, for each polygon $K \in \mathcal{T}_h$, we have that $\mathbf{u}_I - \mathbf{u}_c = \mathbf{0}$ on $\partial K$,
see \eqref{eq:stokesinterpolant}. Hence, it holds
\[
\|\mathbf{u}_I - \mathbf{u}_c\|_{0,K} \leq C \,h_K |\mathbf{u}_I - \mathbf{u}_c|_{1,K}.
\]
Therefore, we get
\begin{equation}\label{eq:locest2}
\begin{split}
\|\mathbf{u} - \mathbf{u}_I\|_{0,K} & \leq \|\mathbf{u} - \mathbf{u}_c\|_{0,K} + \|\mathbf{u}_c - \mathbf{u}_I\|_{0,K} 
\leq C  \left(h_K^{s+1} |\mathbf{u}|_{s+1,D(K)}  + h_K |\mathbf{u}_I - \mathbf{u}_c |_{1,K} \right) \\
&  \leq \left( h_K^{s+1} |\mathbf{u}|_{s+1,D(K)}  + h_K  |\mathbf{u} - \mathbf{u}_I |_{1,K}  + h_K 
|\mathbf{u} - \mathbf{u}_c|_{1,K} \right)  \leq C \, h^{s+1}_K |\mathbf{u} |_{s+1,D(K)}.
\end{split}
\end{equation}
From \eqref{eq:locest1} and \eqref{eq:locest2}, we infer estimate \eqref{eq:interpolation}. 
\end{proof}


\subsection{A stability result: the inf-sup condition}
\label{sub:4.1}

Aim of this section is to prove that the following inf-sup condition holds.
\begin{proposizione}
\label{thm2} Given the discrete spaces
$\mathbf{V}_h$ and $Q_h$ defined in~\eqref{eq:V_h} and~\eqref{eq:Q_h}, there exists a positive $\tilde{\beta}$, independent of $h$, such that:
\begin{equation}
\label{eq:inf-sup discreta}
\sup_{\mathbf{v}_h \in \mathbf{V}_h \, \mathbf{v}_h \neq \mathbf{0}} \frac{b(\mathbf{v}_h, q_h)}{ \|\mathbf{v}_h\|_{1}} \geq \tilde{\beta} \|q_h\|_Q \qquad \text{for all $q_h \in Q_h$.}
\end{equation}
\end{proposizione}

\begin{proof}
We only sketch the proof, because it essentially follows the guidelines of Theorem~3.1 in \cite{VEM-elasticity}. Since the continuous inf-sup condition \eqref{eq:inf-sup} is fulfilled,  it is sufficient to construct a linear operator $\pi_{h} \colon \mathbf{V} \to \mathbf{V}_h$, satisfying (see~\cite{BoffiBrezziFortin}):
\begin{equation}
\label{eq:fortin}
\left\{
\begin{aligned}
&b(\pi_h \mathbf{v}, q_h) = b(\mathbf{v}, q_h)  & &\forall\,\mathbf{v} \in \mathbf{V}, \forall\,q_h \in Q_h, \\
&  \|\pi_h \mathbf{v}\|_{1} \leq c_{\pi}\|\mathbf{v}\|_{1} & &\forall\,\mathbf{v} \in \mathbf{V},
\end{aligned}
\right.
\end{equation} 
where $c_{\pi}$ is a positive $h$-independent constant. 
Given $\mathbf{v} \in \mathbf{V}$, using arguments borrowed from \cite{VEM-elasticity} and considering the VEM interpolant $\mathbf{v}_I$ 
presented in Proposition~\eqref{thm:interpolation}, we first construct $\bar{\mathbf{v}}_h \in \mathbf{V}_h$ such that
\[
 b(\mathbf{v} - \bar{\mathbf{v}}_h, \bar{q}_h) = 0 \qquad \text{$\forall\,\bar{q}_h$ piecewise constant function in $\mathcal{T}_h$}
\]
and
 
\begin{equation}
\label{eq: stime bar v}
\|\mathbf{v} - \bar{\mathbf{v}}_h \|_{1} \leq C  \| {\mathbf{v}} \|_{1} \qquad \forall\, \mathbf{v}\in \mathbf{V}. 
\end{equation}

Next, we build a  ``bubble'' function $\widetilde{\mathbf{v}}_h \in \mathbf{V}_h$, locally defined as follows. Given $K\in\mathcal{T}_h $, we set all the degrees of freedom $\mathbf{D_V1}$, $\mathbf{D_V2}$ and $\mathbf{D_V3}$ equal to zero, while we set the degrees of freedom $\mathbf{D_V4}$ imposing
\begin{equation}
\label{eq:tilde v}
b^K( \widetilde{\mathbf{v}}_h , q_k)= b^K (\mathbf{v} - \bar{\mathbf{v}}_h, q_k) \qquad \forall\, q_k \in \Pk_{k-1}(K).
\end{equation}
It holds:
\begin{equation}
\label{eq: stime tilde v}
\|\widetilde{\mathbf{v}}_h \|_{1} \leq C  \|\mathbf{v} - \bar{\mathbf{v}}_h \|_{1} \leq C  \|{\mathbf{v}} \|_{1}  .
\end{equation}
Now we set
\[
\pi_h \mathbf{v} := \bar{\mathbf{v}}_h + \widetilde{\mathbf{v}}_h \qquad \text{for all $\mathbf{v} \in \mathbf{V}$.}
\]
By \eqref{eq:tilde v}, we have
\[
b(\mathbf{v} - \pi_h \mathbf{v}, q_h) = b(\mathbf{v} - \bar{\mathbf{v}}_h, q_h) - b(\widetilde{\mathbf{v}}_h, q_h) = 0 \qquad \text{for all $q_h\in Q_h$,}
\]
and combining \eqref{eq: stime bar v} and \eqref{eq: stime tilde v}, we get
\[
\|\pi_h \mathbf{v}\|_{1} = \|\bar{\mathbf{v}}_h + \widetilde{\mathbf{v}}_h\|_{1} \leq \|\bar{\mathbf{v}}_h - \mathbf{v}\|_{1} + \|\mathbf{v}\|_{1}+ \|\widetilde{\mathbf{v}}_h\|_{1} \leq  C\|\mathbf{v}\|_{1}.
\]
\end{proof}
An immediate consequence of the previous result is the following Theorem.

\begin{teorema}
Problem \eqref{eq:stokes virtual} has a unique solution $(\mathbf{u}_h, p_h) \in \mathbf{V}_h \times Q_h$, verifying the estimate
\[
\|\mathbf{u}_h\|_{1} + \|p_h\|_{Q} \leq C \|\mathbf{f}\|_{0}.
\]
\end{teorema}
Moreover, the inf-sup condition of Proposition~\ref{thm2}, along with property~\eqref{eq:divfree}, implies that:
\begin{equation}\label{eq:divfree2}
{\rm div}\, \mathbf{V}_h = Q_h .
\end{equation}

\begin{osservazione}
\label{oss:inf-supclassic}
An analogous result of Proposition \ref{thm2} is shown in \cite{VEM-elasticity}, where the discrete inf-sup condition is detailed for the virtual local spaces defined in Remark \ref{oss:3.1}.  Therefore, as already observed, also the spaces of \cite{VEM-elasticity} could be directly used as a stable pair for the Stokes problem. On the other hand, the choice in \cite{VEM-elasticity} would not satisfy condition
\eqref{kernincl} and thus the discrete solution would not be divergence free. Moreover, such spaces would not share the interesting property 
to be equivalent to a suitable reduced problem (cf. Section \ref{sec:5}).
\end{osservazione}

%

\subsection{A convergence result}
\label{sub:4.2}

We begin by remarking that, using Proposition \ref{thm:interpolation} and classical approximation theory, for $\mathbf{v} \in [H^{k+1}(\Omega)]^2$ and $q \in H^k(\Omega)$ it holds
 \begin{equation}
 \label{eq: stime v_I}
 \inf_{ \mathbf{v}_h\in \mathbf{V}_h }\|\mathbf{v} - \mathbf{v}_h\|_{1} \leq C h^k |\mathbf{v}|_{k+1}
 \end{equation}
 and 
 \begin{equation}
 \label{eq: stime q_I}
 \inf_{ \mathbf{q}_h\in \mathbf{Q}_h }\|q - q_h \|_{Q} \leq C h^k |q|_k.
 \end{equation}

We now notice that, if $\mathbf{u}\in\mathbf{V}$ is the velocity solution to Problem \eqref{eq:stokes variazionale}, then it is the solution to Problem (cf. also~\eqref{eq:Z}):
\begin{equation}
\label{eq:stokes nucleo}
\left\{
\begin{aligned}
& \text{find $\mathbf{u} \in \mathbf{Z}$ } \\
& a(\mathbf{u}, \mathbf{v}) = (\mathbf{f}, \mathbf{v}) \qquad & \text{for all $\mathbf{v} \in \mathbf{Z}$.} 
\end{aligned}
\right.
\end{equation}
Analogously, if $\mathbf{u}_h\in\mathbf{V}_h$ is the velocity solution to Problem \eqref{eq:stokes virtual}, then it is the solution to Problem (cf. also~\eqref{eq:Z_h}):
\begin{equation}
\label{eq:stokes nucleo discreto}
\left\{
\begin{aligned}
& \text{find $\mathbf{u}_h \in \mathbf{Z}_h$ } \\
& a_h(\mathbf{u}_h, \mathbf{v}_h) = (\mathbf{f}_h, \mathbf{v}_h) \qquad & \text{for all $\mathbf{v}_h \in \mathbf{Z}_h$,} 
\end{aligned}
\right.
\end{equation}
Recalling~\eqref{kernincl}, Problem \eqref{eq:stokes nucleo discreto} can be seen as a standard virtual approximation of the elliptic problem \eqref{eq:stokes nucleo}. 
Furthermore, given $\mathbf{z}\in \mathbf{Z}$, the inf-sup condition~\eqref{eq:inf-sup discreta} implies (see~\cite{BoffiBrezziFortin}):
\begin{equation*}
\inf_{\mathbf{z}_h\in \mathbf{Z}_h} ||\mathbf{z} -\mathbf{z}_h ||_{1}\le C\inf_{\mathbf{v}_h\in \mathbf{V}_h} ||\mathbf{z} -\mathbf{v}_h ||_{1},
\end{equation*}  
which essentially means that $\mathbf{Z}$ is approximated by $\mathbf{Z}_h$ with the same accuracy order of the whole subspace $ \mathbf{V}_h$.
As a consequence, usual VEM arguments (for instance, as in \cite{volley}) and~\eqref{eq: stime v_I} lead to the following result.

\begin{teorema}
\label{thm5}
Let $\mathbf{u} \in \mathbf{Z}$ be the solution of problem \eqref{eq:stokes nucleo} and $\mathbf{u}_h \in \mathbf{Z}_h$ be the solution of problem \eqref{eq:stokes nucleo discreto}. Then
\[
\|\mathbf{u} - \mathbf{u}_h\|_{1} \leq Ch^k \left( |\mathbf{f}|_{k-1} + |\mathbf{u}|_{k+1}\right).
\]  
\end{teorema}

We proceed by analysing the error on the pressure field. We are ready to prove the following error estimates for the pressure approximation.

\begin{teorema}
\label{thm4}
Let $(\mathbf{u}, p) \in \mathbf{V} \times Q$ be the solution of Problem \eqref{eq:stokes variazionale} and $(\mathbf{u}_h, p_h) \in \mathbf{V}_h \times Q_h$ be the solution of Problem \eqref{eq:stokes virtual}. Then it holds:
\begin{equation}\label{eq:pressure-estimate}
\|p -p_h\|_Q \leq C h^k \left( |\mathbf{f}|_{k-1} + |\mathbf{u}|_{k+1} + |p|_{k} \right).
\end{equation}
\end{teorema}

\begin{proof}
Let $q_h \in Q_h$. From the discrete inf-sup condition \eqref{eq:inf-sup discreta}, we infer:
\begin{equation}
\label{eq:thm4.1}
\tilde{\beta} \|p_h - q_h\|_Q \leq \sup_{\mathbf{v}_h \in \mathbf{V}_h \, \mathbf{v}_h \neq \mathbf{0}} \frac{b(\mathbf{v}_h, p_h - q_h)}{\|\mathbf{v}_h\|_{1}} = \sup_{\mathbf{v}_h \in \mathbf{V}_h \, \mathbf{v}_h \neq \mathbf{0}} \frac{b(\mathbf{v}_h, p_h - p) + b(\mathbf{v}_h,  p - q_h)}{\|\mathbf{v}_h\|_{1}}. 
\end{equation}
Since $(\mathbf{u},p)$ and $(\mathbf{u}_h,p)$ are the solution of \eqref{eq:stokes variazionale} and \eqref{eq:stokes virtual}, respectively, it follows that
\begin{gather*}
a(\mathbf{u}, \mathbf{v}_h) + b(\mathbf{v}_h, p) = (\mathbf{f}, \mathbf{v}_h) \qquad  \text{for all $\mathbf{v}_h \in \mathbf{V}_h$,} \\
a_h(\mathbf{u}_h, \mathbf{v}_h) + b(\mathbf{v}_h, p_h) = (\mathbf{f}_h, \mathbf{v}_h) \qquad  \text{for all $\mathbf{v}_h \in \mathbf{V}_h$.} 
\end{gather*}
Therefore, we get
\begin{equation}
\label{eq:thm4.2}
b(\mathbf{v}_h, p_h - p)  = (\mathbf{f}_h - \mathbf{f}, \mathbf{v}_h) + \left(a(\mathbf{u}, \mathbf{v}_h) - a_h(\mathbf{u}_h, \mathbf{v}_h)  \right) =: \mu_1(\mathbf{v}_h) + \mu_2(\mathbf{v}_h) \qquad  \text{for all $\mathbf{v}_h \in \mathbf{V}_h$.}
\end{equation}
The term $ \mu_1(\mathbf{v}_h)$ can be bounded by making use of Lemma \ref{lemma1}:
\begin{equation}
\label{eq:mu_1bis}
|\mu_1(\mathbf{v}_h)| \leq Ch^k |\mathbf{f}|_{k-1} \|\mathbf{v}_h\|_{1}.
\end{equation}
For the term $ \mu_2(\mathbf{v}_h)$, using \eqref{eq:consist} and the continuity of $a_h(\cdot, \cdot)$ and the triangle inequality, we get:
\[
\begin{split}
\mu_2(\mathbf{v}_h) &= a(\mathbf{u}, \mathbf{v}_h)  - a_h(\mathbf{u}_h, \mathbf{v}_h)  
= \sum_{K \in \mathcal{T}_h} \biggl( a^K(\mathbf{u}, \mathbf{v}_h)  - a_h^K(\mathbf{u}_h, \mathbf{v}_h)\biggr) \\
&= \sum_{K \in \mathcal{T}_h} \biggl( a^K(\mathbf{u} - \mathbf{u}_{\pi}, \mathbf{v}_h)  + 
a_h^K(\mathbf{u}_{\pi} - \mathbf{u}  + \mathbf{u} -\mathbf{u}_h, \mathbf{v}_h) \biggr) \\
& \leq \sum_{K \in \mathcal{T}_h} C \bigl( |\mathbf{u} - \mathbf{u}_{\pi}|_{1,K} + 
|(\mathbf{u}_{\pi} - \mathbf{u})  + (\mathbf{u} -\mathbf{u}_h)|_{1,K}  \bigr) |\mathbf{v}_h|_{1,K} \\
&  \leq  \sum_{K \in \mathcal{T}_h} C  \bigl( |\mathbf{u} - \mathbf{u}_{\pi}|_{1,K} + |\mathbf{u} -\mathbf{u}_h|_{1,K}  \bigr) 
|\mathbf{v}_h|_{1,K} 
\end{split}
\]
where $\mathbf{u}_{\pi}$  is the piecewise polynomial of degree $k$ defined in Lemma \ref{lm:scott}. Then, from estimate~\eqref{eq:scott} and 
Theorem \ref{thm5}, we obtain
\begin{equation}
\label{eq:mu_2bis}
\begin{split}
|\mu_2(\mathbf{v}_h)| &\leq  C h^k \left( |\mathbf{f}|_{k-1} + |\mathbf{u}|_{k+1}\right) \|\mathbf{v}_h\|_{1}.
\end{split}
\end{equation}
Then, combining \eqref{eq:mu_1bis} and \eqref{eq:mu_2bis} in \eqref{eq:thm4.2}, we get
\begin{equation}
\label{eq:thm4.3}
|b(\mathbf{v}_h, p_h - p) | \leq C h^k \left( |\mathbf{f}|_{k-1} + |\mathbf{u}|_{k+1}  \right) \|\mathbf{v}_h\|_{1} .
\end{equation}
Moreover, we have
\begin{equation}
\label{eq:thm4.4}
|b(\mathbf{v}_h, p - q_h) | \leq C \|p - q_h\|_{Q}\|\mathbf{v}_h\|_{1}.
\end{equation}
Then, using \eqref{eq:thm4.3} and \eqref{eq:thm4.4} in \eqref{eq:thm4.1}, we infer

\begin{equation}\label{eq:press-est}
\|p_h - q_h\|_Q \leq C h^k \left( |\mathbf{f}|_{k-1} + |\mathbf{u}|_{k+1} \right) + C  \|p - q_h\|_{Q}.
\end{equation}
Finally, using~\eqref{eq:press-est} and the triangular inequality, we get
\[
\|p -p_h\|_Q \leq \|p -q_h\|_Q + \|p_h - q_h\|_Q \leq C h^k \left( |\mathbf{f}|_{k-1} + |\mathbf{u}|_{k+1}  \right) + C  \|p - q_h\|_{Q} \qquad\forall\,q_h \in Q_h .
\]
Passing to the infimum with respect to $q_h \in Q_h$, and using estimate \eqref{eq: stime q_I}, we obtain~\eqref{eq:pressure-estimate}.
\end{proof}


\section{Reduced spaces and reduced problem}
\label{sec:5}

In this section we show that Problem~\eqref{eq:stokes virtual} is somehow equivalent to a suitable reduced problem (cf. Proposition~\ref{thm6}), involving significant fewer degrees of freedom, especially for large $k$.
Let us define the reduced local virtual spaces:
\begin{multline}
\label{eq:V_h^Kbis}
\widehat{\mathbf{V}}_h^K := \biggl\{  
\mathbf{v} \in [H^1(K)]^2 \quad \text{s.t} \quad \mathbf{v}_{|{\partial K}} \in [\B_k(\partial K)]^2 \, , \biggr.
\\
\left.
\biggl\{
\begin{aligned}
& - \nu \, \boldsymbol{\Delta}    \mathbf{v}  -  \nabla s \in \mathcal{G}_{k-2}(K)^{\perp},  \\
& {\rm div} \, \mathbf{v} \in \Pk_{0}(K),
\end{aligned}
\biggr. \qquad \text{for some $s \in H^1(K)$}
\quad \right\}
\end{multline}
and
\begin{equation}
\label{eq:Q_h^Kbis}
\widehat{Q}_h^K :=  \Pk_{0}(K).
\end{equation}
Moreover, we have:
\begin{equation}
\label{eq:dimensione V_h^Kbis}
\dim\left( \widehat{\mathbf{V}}_h^K \right) = \dim\left([\B_k(\partial K)]^2\right) + \dim\left(\mathcal{G}_{k-2}(K)^{\perp}\right) = 2n_K k + \frac{(k-1)(k-2)}{2},
\end{equation}
and
\begin{equation}
\label{eq:dimensione Q_h^Kbis}
\dim(\widehat{Q}_h^K) = \dim(\Pk_{0}(K))  = 1,
\end{equation}
where $n_K$ is the number of edges in $\partial K$. As sets of degrees of freedom for the reduced spaces, we may consider the following. 

For every function $\mathbf{v} \in \widehat{\mathbf{V}}_h^K$ we take
\begin{itemize}
\item $\mathbf{\widehat{D}_V1}$: the values of $\mathbf{v}$ at each vertex of the polygon $K$,
\item $\mathbf{\widehat{D}_V2}$: the values of $\mathbf{v}$ at $k-1$ distinct points of every edge $e \in \partial K$,
\item $\mathbf{\widehat{D}_V3}$: the moments
\[
\int_K \mathbf{v} \cdot \mathbf{g}_{k-2}^{\perp} \, {\rm d}K \qquad \text{for all $\mathbf{g}_{k-2}^{\perp} \in \mathcal{G}_{k-2}(K)^{\perp}$.}
\]
\end{itemize} 

For every $q \in \widehat{Q}_h$ we consider
\begin{itemize}
\item $\mathbf{\widehat{D}_Q}$: the moment
\[
\int_K q  \, {\rm d}K.
\]
\end{itemize}
We define the global reduced virtual element spaces by setting
\begin{equation}
\label{eq:V_hbis}
\widehat{\mathbf{V}}_h := \{ \mathbf{v} \in [H^1_0(\Omega)]^2 \quad \text{s.t} \quad \mathbf{v}_{|K} \in \widehat{\mathbf{V}}_h^K  \quad \text{for all $K \in \mathcal{T}_h$} \}
\end{equation} 
and
\begin{equation}
\label{eq:Q_hbis}
\widehat{Q}_h := \{ q \in L_0^2(\Omega) \quad \text{s.t.} \quad q_{|K} \in  \widehat{Q}_h^K \quad \text{for all $K \in \mathcal{T}_h$}\}.
\end{equation}
It is easy to check that
\begin{equation}\label{eq:Vdofs-red}
\dim(\widehat{\mathbf{V}}_h) = n_P  \frac{(k-1)(k-2)}{2}+  2 (n_V + (k-1) n_E)
\end{equation}
and
\begin{equation}\label{eq:Qdofs-red}
\dim(\widehat{Q}_h) = n_P - 1
\end{equation}
where we recall that $n_P$ is the number of elements in $\mathcal{T}_h$, $n_E$ and $n_V$ are respectively the number of internal edges and internal vertexes in the decomposition. 

The reduced virtual element discretization of the Stokes problem \eqref{eq:stokes variazionale} is then:
\begin{equation}
\label{eq:stokes reduced virtual}
\left\{
\begin{aligned}
& \text{find $\widehat{\mathbf{u}}_h \in \widehat{\mathbf{V}}_h$ and $\widehat{p}_h \in \widehat{Q}_h$, such that} \\
& a_h(\widehat{\mathbf{u}}_h, \widehat{\mathbf{v}}_h) + b(\widehat{\mathbf{v}}_h, \widehat{p}_h) = (\mathbf{f}_h, \widehat{\mathbf{v}}_h) \qquad & \text{for all $\widehat{\mathbf{v}}_h \in \widehat{\mathbf{V}}_h$,} \\
&  b(\widehat{\mathbf{u}}_h, \widehat{q}_h) = 0 \qquad & \text{for all $\widehat{q}_h \in \widehat{Q}_h$.}
\end{aligned}
\right.
\end{equation}
Above, the bilinear forms $a_h(\cdot, \cdot)$  and $b(\cdot, \cdot)$, and the loading term $ \mathbf{f}_h $  are the same as before, see~\eqref{eq:a_h}, \eqref{bhform} and~\eqref{eq:f_h}. It is easily seen that all the terms involved in~\eqref{eq:stokes reduced virtual} are computable by means of the new reduced degrees of freedom. For example, to compute $(\mathbf{f}_h, \widehat{\mathbf{v}}_h)$ one needs to compute $\Pi_{k-2}^{0, K}  \widehat{\mathbf{v}}_h $, see~\eqref{eq:right}. 
However,  for any $\mathbf{q}_{k-2} \in [\Pk_{k-2}(K)]^2$ we have:  
\[
\int_K  \Pi_{k-2}^{0, K}  \widehat{\mathbf{v}}_h\cdot \mathbf{q}_{k-2} \, {\rm d}K = \int_K    \widehat{\mathbf{v}}_h\cdot \mathbf{q}_{k-2} \, {\rm d}K = 
\int_K    \widehat{\mathbf{v}}_h\cdot \nabla q_{k-1} \, {\rm d}K  + \int_K    \widehat{\mathbf{v}}_h\cdot \mathbf{g}_{k-2}^{\perp} \, {\rm d}K 
\]
for suitable $q_{k-1} \in \Pk_{k-1}(K)$ and $\mathbf{g}_{k-2}^{\perp} \in \mathcal{G}_{k-2}(K)^{\perp}$. 
Then, since ${\rm div} \, \widehat{\mathbf{v}}_h \in \Pk_0(K)$, denoting with $|K|$ the area of $K$, we get
\[
\begin{split}
\int_K  \Pi_{k-2}^{0, K}  \widehat{\mathbf{v}}_h\cdot \mathbf{q}_{k-2} \, {\rm d}K &= -\int_K  {\rm div} \, \widehat{\mathbf{v}}_h \, q_{k-1} \, {\rm d}K + \int_{\partial K} q_{k-1} \widehat{\mathbf{v}}_h \cdot \mathbf{n} \, {\rm d}s + \int_K   \widehat{\mathbf{v}}_h\cdot \mathbf{g}_{k-2}^{\perp} \, {\rm d}K \\
& = -|K|^{-1}\left( \int_{\partial K} \widehat{\mathbf{v}}_h \cdot \mathbf{n} \,{\rm d}s \right)  \int_K    q_{k-1} \, {\rm d}K + \int_{\partial K} q_{k-1} \widehat{\mathbf{v}}_h \cdot \mathbf{n} \, {\rm d}s + \int_K   \widehat{\mathbf{v}}_h\cdot \mathbf{g}_{k-2}^{\perp} \, {\rm d}K 
\end{split}
\]
whose right-hand side is directly computable from $\mathbf{\widehat{D}_V}$.

In addition, using the same techniques of Proposition~\ref{thm2} (take $\pi_h {\bf v} = \bar {\bf v}_h$ in the proof), one can prove that

\begin{equation}
\label{eq:inf-sup reduced discreta}
\exists \, \widehat{\beta} >0 \quad \text{such that} \quad \sup_{\widehat{\mathbf{v}}_h \in \widehat{\mathbf{V}}_h \, \widehat{\mathbf{v}}_h \neq \mathbf{0}} \frac{b(\widehat{\mathbf{v}}_h, \widehat{q}_h)}{ \|\widehat{\mathbf{v}}_h\|_{1}} \geq \widehat{\beta} \|\widehat{q}_h\|_Q \qquad \text{for all $\widehat{q}_h \in \widehat{Q}_h$} .
\end{equation}

The following proposition states the relation between Problem~\eqref{eq:stokes virtual} and the reduced Problem~\eqref{eq:stokes reduced virtual}.
\begin{proposizione}
\label{thm6}
Let $(\mathbf{u}_h, p_h) \in \mathbf{V}_h \times Q_h$ be the solution of problem \eqref{eq:stokes virtual} and $(\widehat{\mathbf{u}}_h, \widehat{p}_h) \in \widehat{\mathbf{V}}_h \times \widehat{Q}_h$ be the solution of problem \eqref{eq:stokes reduced virtual}. Then
\begin{equation}\label{eq:equiv}
\widehat{\mathbf{u}}_h = \mathbf{u}_h   \qquad \text{and} \qquad \widehat{p}_{h|K} = \Pi_0^{0, K} p_h \quad \text{for all $K \in \mathcal{T}_h$.}
\end{equation}
\end{proposizione}

\begin{proof}
Let 
\[
\widehat{\mathbf{Z}}_h := \{ \widehat{\mathbf{v}}_h \in \widehat{\mathbf{V}}_h \quad \text{s.t.} \quad b(\widehat{\mathbf{v}}_h, \widehat{q}_h) = 0 \quad \text{for all $\widehat{q}_h \in \widehat{Q}_h$}\}.
\]
Then $\widehat{\mathbf{u}}_h $ solves (cf.~\eqref{eq:stokes nucleo discreto}):
\begin{equation}
\label{eq:red-kern}
\left\{
\begin{aligned}
& \text{find $\widehat{\mathbf{u}}_h \in \widehat{\mathbf{Z}}_h$ } \\
& a_h(\widehat{\mathbf{u}}_h, \widehat{\mathbf{v}}_h) = (\mathbf{f}_h, \widehat{\mathbf{v}}_h) \qquad & \text{for all $\widehat{\mathbf{v}}_h \in \widehat{\mathbf{Z}}_h$. } 
\end{aligned}
\right.
\end{equation}

We now notice that  $\widehat{\mathbf{Z}}_h = {\mathbf{Z}}_h$, see~\eqref{eq:Z_h}. Therefore, Problem \eqref{eq:red-kern} is equivalent to Problem \eqref{eq:stokes nucleo discreto} and $\widehat{\mathbf{u}}_h = \mathbf{u}_h$.

For the pressure component of the solution, from \eqref{eq:stokes virtual} and \eqref{eq:stokes reduced virtual}, we get
\begin{gather}
b(\mathbf{v}_h, p_h) = (\mathbf{f}_h, \mathbf{v}_h) - a_h(\mathbf{u}_h, \mathbf{v}_h) \qquad \text{for all $\mathbf{v}_h \in \mathbf{V}_h$}\label{eq:pheq}\\
b(\widehat{\mathbf{v}}_h, \widehat{p}_h) = (\mathbf{f}_h, \widehat{\mathbf{v}}_h) - a_h(\widehat{\mathbf{u}}_h, \widehat{\mathbf{v}}_h) \qquad \text{for all $\widehat{\mathbf{v}}_h \in \widehat{\mathbf{V}}_h$.}\label{eq:pheqred}
\end{gather}
Let $p_h =: p_0 + p^{\perp}$, 
where $p_{0|K} = \Pi_0^{0,K} p_h$ for all $K \in \mathcal{T}_h$, and $p^{\perp}:= p_h - p_0$ (hence $\int_K  p^{\perp} \, {\rm d}K =0$). From \eqref{eq:pheq}, we have
\[
b(\mathbf{v}_h,p_0 + p^{\perp}) = (\mathbf{f}_h, \mathbf{v}_h) - a_h(\mathbf{u}_h, \mathbf{v}_h) \qquad \text{for all $\mathbf{v}_h \in \mathbf{V}_h$}.
\]
Since $\widehat{\mathbf{V}}_h\subseteq \mathbf{V}_h$, we deduce
\[
b(\widehat{\mathbf{v}}_h,p_0 ) + b(\widehat{\mathbf{v}}_h,  p^{\perp}) = (\mathbf{f}_h, \widehat{\mathbf{v}}_h) - a_h(\mathbf{u}_h, \widehat{\mathbf{v}}_h) \qquad \text{for all $\widehat{\mathbf{v}}_h \in \widehat{\mathbf{V}}_h$.}
\]
Now, $b(\widehat{\mathbf{v}}_h,  p^{\perp})=0$ because ${\rm div}\, \widehat{\mathbf{v}}_h$ is constant on each polygon $K$. We conclude that
\begin{equation}\label{eq:almost-done}
b(\widehat{\mathbf{v}}_h,p_0 )  = (\mathbf{f}_h, \widehat{\mathbf{v}}_h) - a_h(\mathbf{u}_h, \widehat{\mathbf{v}}_h) \qquad \text{for all $\widehat{\mathbf{v}}_h \in \widehat{\mathbf{V}}_h$.}
\end{equation} 
From \eqref{eq:almost-done} and recalling that $\mathbf{u}_h = \widehat{\mathbf{u}}_h $, we get that $(\widehat{\mathbf{u}}_h, p_0)\in \widehat{\mathbf{V}}_h \times \widehat{Q}_h $ solves Problem \eqref{eq:stokes reduced virtual}. Uniqueness of the solution of Problem~\eqref{eq:stokes reduced virtual} then implies $ \widehat{p}_{h|K} = p_{0|K}$ for every $K$, and \eqref{eq:equiv} is proved.

\end{proof}

\begin{osservazione}
\label{oss:dofs}
Proposition \ref{thm6} allows us to solve the Stokes Problem \eqref{eq:stokes variazionale} directly by making use of the reduced problem \eqref{eq:stokes reduced virtual}, saving $n_P ((k+1)k - 2)$ degrees of freedom, see~\eqref{eq:Vdofs}, \eqref{eq:Qdofs}, \eqref{eq:Vdofs-red} and~\eqref{eq:Qdofs-red}. In Table \ref{tabledof} we display this quantity (with respect the total amount of the original DoFs) for the sequences of meshes introduced in Section \ref{sec:6} with $k=2,3,4,5$ in order to have an estimate of the saving in the reduced linear system with respect its original size.
\begin{table}[!h]
\centering
\begin{tabular}{ll*{4}{c}}
\toprule
& &$k=2$ & $k=3$ & $k=4$ & $k=5$\\
\midrule
\multirow{4}*{$\mathcal{V}_h$}
& $h = 1/4$   & $34.408\%$ & $43.715\%$ &  $48.484\%$ &  $51.494\%$ \\
& $h = 1/8$   & $30.260\%$ & $39.506\%$ &  $44.547\%$ &  $47.863\%$ \\
& $h = 1/16$ & $28.460\%$ & $37.624\%$ &  $42.753\%$ &  $46.185\%$  \\
& $h = 1/32$ & $27.634\%$ & $36.749\%$ &  $41.911\%$ &  $45.392\%$ \\
\midrule
\multirow{4}*{$\mathcal{T}_h$}
& $h = 1/2$   & $49.230\%$ & $56.737\%$ &  $59.751\%$ &  $61.369\%$ \\
& $h = 1/4$   & $47.761\%$ & $55.427\%$ &  $58.616\%$ &  $60.377\%$ \\
& $h = 1/8$  & $45.937\%$ & $53.889\%$ &  $57.314\%$ &  $59.253\%$  \\
& $h = 1/16$ & $45.171\%$ & $53.243\%$ &  $56.767\%$ &  $58.780\%$ \\
\midrule
\multirow{4}*{$\mathcal{Q}_h$}
& $h = 1/4$   & $43.835\%$ & $52.287\%$ &  $56.031\%$ &  $58.181\%$ \\
& $h = 1/8$   & $39.875\%$ & $48.706\%$ &  $52.892\%$ &  $55.411\%$ \\
& $h = 1/16$ & $38.066\%$ & $47.041\%$ &  $51.417\%$ &  $54.098\%$  \\
& $h = 1/32$ & $37.202\%$ & $46.238\%$ &  $50.701\%$ &  $53.458\%$ \\
\bottomrule
\end{tabular}
\caption{Percentage saving of DoFs in the reduced problem with respect the original one.}
\label{tabledof}
\end{table}
\end{osservazione}

In addition, we remark that Proposition~\ref{thm6} holds not only when homogeneous Dirichlet conditions are applied on the whole boundary, 
but also for other (possibly non-homogeneous) boundary conditions, as numerically shown in Section~\ref{sec:6}.

\begin{osservazione}
\label{oss:equiv}
It is possible to give an alternative proof of Proposition \ref{thm6}  directly in terms of the associated linear system. 
Furthermore, it is also possible to  implement the  ``reduced'' Problem \eqref{eq:stokes reduced virtual} by coding the 
``complete'' Stokes Problem \eqref{eq:stokes virtual} and locally removing 
the rows and the columns relative to the extra degrees of freedom.
We detail these aspects in the next section.
\end{osservazione}

\begin{osservazione}
\label{oss:postproc}
Given the solution of \eqref{eq:stokes reduced virtual}, if one is interested in a more accurate pressure, the discrete scalar field $p_h$ 
can be recovered by an element-wise post processing procedure. Such local problems can be, for instance, immediately extracted from the removed 
rows and columns mentioned in Remark \ref{oss:equiv}.  
\end{osservazione}

\subsection{Algebraic aspects}
\label{sub:preprin}

It is possible to get the result of Proposition \ref{thm6} also using a matrix point of view. 
Let us first introduce $\boldsymbol{\alpha}:=(\alpha_1, \alpha_2)$ and $|\boldsymbol{\alpha}| = \alpha_1 + \alpha_2$.   
We denote with  $\{m^K_{\boldsymbol{\alpha}}\}$, for all $K \in \mathcal{T}_h$ and for $|\boldsymbol{\alpha}|=0, \dots, k-1$, 
an orthonormal basis in $L^2(K)$ of $\Pk_{k-1}(K)$. Then 
\[
\mathcal{M}_{k-1} := \{m^K_{\boldsymbol{\alpha}} \quad \text{for $|\boldsymbol{\alpha}|=0, \dots, k-1$ and  $K \in \mathcal{T}_h$ }\}
\] 
is an orthonormal basis for $Q_h$ and 
\[
\widehat{\mathcal{M}}_{0} := \{m^K_{\boldsymbol{0}} \quad \text{for $K \in \mathcal{T}_h$}\}
\] 
is a basis for $\widehat{Q}_h$. 

Now, introduce the global basis functions:
\begin{itemize}
\item $\boldsymbol{\phi}^{i_1}_1$ for $i_1=1, \dots, 2n_V$, which correspond to the degrees of freedom of type $\mathbf{D_V1}$, 
\item $\boldsymbol{\phi}^{i_2}_2$ for $i_2=1, \dots, 2(k-1)n_E$, which correspond to the degrees of freedom of type $\mathbf{D_V2}$,  
\item $\boldsymbol{\phi}^{i_3}_3$ for $i_3=1, \dots, n_P\frac{(k-1)(k-2)}{2}$, which correspond to the degrees of freedom of type $\mathbf{D_V3}$,   
\item $\boldsymbol{\phi}^{i_4}_4$ for $i_4=1, \dots, n_P\left(\frac{(k+1)k}{2} - 1\right)$  which correspond to the degrees of freedom 
of type $\mathbf{D_V4}$, selecting 
\[
\int_K {\rm div} \,\mathbf{v} \, m^K_{\boldsymbol{\alpha}} \, {\rm d}K \qquad \text{for all $m^K_{\boldsymbol{\alpha}} \in \mathcal{M}_{k-1} / \widehat{\mathcal{M}}_{0} $.}
\] 
for all $K \in \mathcal{T}_h$.
\end{itemize}
Then 
\[
\{\boldsymbol{\phi}^{i_1}_1, \boldsymbol{\phi}^{i_2}_2, \boldsymbol{\phi}^{i_3}_3, \boldsymbol{\phi}^{i_4}_4 \quad | \quad i_1, i_2, i_3, i_4 \}
\] 
is a basis for $\mathbf{V}_h$ and 
\[
\{\boldsymbol{\phi}^{i_1}_1, \boldsymbol{\phi}^{i_2}_2, \boldsymbol{\phi}^{i_3}_3\quad | \quad i_1, i_2, i_3\}
\]
is a basis for $\widehat{\mathbf{V}}_h$. 
Next, define the global stiffness matrix $\mathbf{S}$ as
\[
(\mathbf{S})_{i_j, i_l} := (S_{j,k})_{i_j, i_l} = a_h(\boldsymbol{\phi}^{i_j}_j, \boldsymbol{\phi}^{i_l}_l) \qquad \text{for $j,l=1,\dots, 4$}
\]
and the  matrix $\mathbf{B}$ as
\[
(\mathbf{B})_{\boldsymbol{\alpha}, i_j} := b(\boldsymbol{\phi}^{i_j}_j, m_{\boldsymbol{\alpha}}) \qquad \text{for $|\boldsymbol{\alpha}|=0, \dots, k$ and $j=1,\dots, 4$.}
\]
Using the definition of $m_{\boldsymbol{\alpha}}$ and $\boldsymbol{\phi}^{i_j}_j$, we observe that the algebraic 
formulation of \eqref{eq:stokes virtual}, reads as follows
\begin{equation}
\label{eq:stokes matrix}
\left[\begin{array}{c c c  c|c c}
S_{11}&S_{12}&S_{13}&S_{14}&B_1^T&0\\

S_{12}^T&S_{22}&S_{23}&S_{24}&B_2^T&0\\

S_{13}^T&S_{23}^T&S_{33}&S_{34}&0&0\\

S_{14}^T&S_{24}^T&S_{34}^T&S_{44}&0&I\\
\hline
B_1&B_2&0&0&0&0\\
 
0&0&0&I&0&0\\
\end{array}\right]
\left[\begin{array}{c}
\mathbf{u}_1 \\
\mathbf{u}_2 \\
\mathbf{u}_3 \\
\mathbf{u}_4 \\
\hline
p_0 \\
p^{\perp} \\
\end{array}\right]
=
\left[\begin{array}{c}
\mathbf{f}_{h,1} \\
\mathbf{f}_{h,2} \\
\mathbf{f}_{h,3} \\
\mathbf{f}_{h,4} \\
\hline
0\\
0\\
\end{array}\right] ,
\end{equation}
whereas the matrix form of problem \eqref{eq:stokes reduced virtual} is   
\begin{equation}
\label{eq:stokes reduced matrix}
\left[\begin{array}{c c  c|c}
S_{11}&S_{12}&S_{13}&B_1^T\\

S_{12}^T&S_{22}&S_{23}&B_2^T\\

S_{13}^T&S_{23}^T&S_{33}&0\\

\hline
B_1&B_2&0&0\\
\end{array}\right]
\left[\begin{array}{c}
\widehat{\mathbf{u}}_1 \\
\widehat{\mathbf{u}}_2 \\
\widehat{\mathbf{u}}_3 \\
\hline
\widehat{p}_0 \\
\end{array}\right]
=
\left[\begin{array}{c}
\mathbf{f}_{h,1} \\
\mathbf{f}_{h,2} \\
\mathbf{f}_{h,3} \\
\hline
0\\
\end{array}\right] .
\end{equation}         
From the last row of \eqref{eq:stokes matrix}  we get $\mathbf{u}_4 = \mathbf{0}$. Therefore \eqref{eq:stokes matrix} 
and \eqref{eq:stokes reduced matrix} are equivalent. In particular 
\[
\mathbf{u}_4 = \mathbf{0} \quad \text{and} \quad \mathbf{u}_j = \widehat{\mathbf{u}}_j \quad \text{for $j=1,2,3$} 
\qquad \text{imply} \qquad \mathbf{u}_h = \widehat{\mathbf{u}}_h.
\]
Moreover, by definition of $\widehat{\mathcal{M}}_0$, $p_0 = \widehat{p}_0$ implies $\widehat{p}_h = \Pi^{0}_0 p_h$.


\section{Numerical tests}
\label{sec:6}

In this section we present two numerical experiments to test the actual performance of the method. In the first test we compare the 
reduced method introduced in Section \ref{sec:5} with the method presented in~\cite{VEM-elasticity} (cf. Remark \ref{oss:3.1}). 
In the second experiment we  investigate numerically the equivalence proved in Proposition \ref{thm6}, considering the more general case of 
non-homogeneous 
boundary conditions.

As usual in the VEM framework, to compute discretization errors we compare the obtained numerical solution with a suitable VEM interpolation of
the analytical solution $\mathbf{u}$. More precisely, we define $\mathbf{u}^I \in \mathbf{V}_h$ by imposing  
\begin{equation}\label{eq:intdef}
\mathbf{D_V}(\mathbf{u}^I) = \mathbf{D_V}(\mathbf{u}) .
\end{equation}
In the same way, we can define the interpolant of $\mathbf{u}$ in $\widetilde{\mathbf{V}}_h$ 
by making use of the degrees of freedom $\mathbf{\widetilde{D}_V}$ (see Remark \ref{oss:3.1}), and in $\widehat{\mathbf{V}}_h$ by making use of 
the degrees of freedom $\mathbf{\widetilde{D}_V}$ (see Section \ref{sec:5}).
Analogously, for the pressure $p$ the interpolant $p^I\in Q_h$ (resp., $\widehat{Q}_h$) is defined by insisting that 
$\mathbf{D_Q}(p^I) = \mathbf{D_Q}(p)$ (resp., $\mathbf{\widehat{D}_Q}(p^I) =\mathbf{\widehat{D}_Q}(p)$).

In our tests  the computational domain is $\Omega= [0,1] ^2$, and it is partitioned using the following sequences of polygonal meshes:
\begin{itemize}
\item $\{ \mathcal{V}_h\}_h$: sequence of Voronoi meshes with $h=1/4, 1/8, 1/16, 1/32$,
\item $\{ \mathcal{T}_h\}_h$: sequence of triangular meshes with $h=1/2, 1/4, 1/8, 1/16$,
\item $\{ \mathcal{Q}_h\}_h$: sequence of square meshes with $h=1/4, 1/8, 1/16, 1/32$.
\end{itemize}
An example of the adopted meshes is shown in Figure \ref{Figure1}.  
\begin{figure}[!h]
\centering
\includegraphics[scale=0.35]{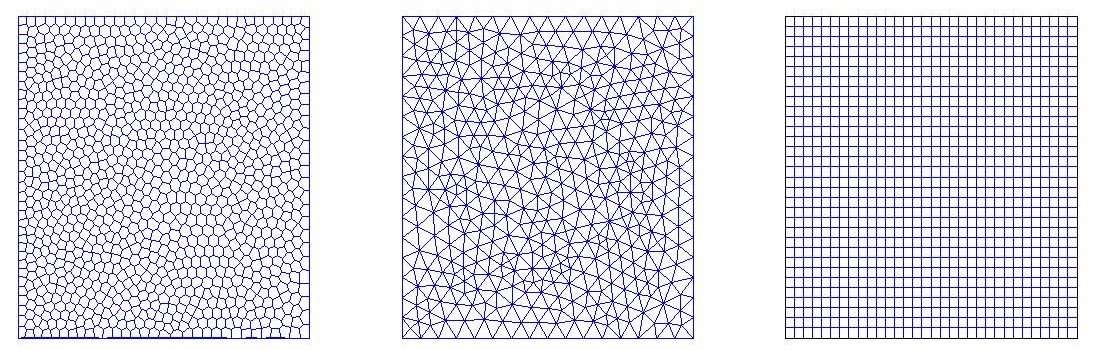}
\caption{Example of polygonal meshes: $\mathcal{V}_{1/32}$,  $\mathcal{T}_{1/16}$, $\mathcal{Q}_{1/32}$.}
\label{Figure1}
\end{figure}
For the generation of the Voronoi meshes we used the code Polymesher \cite{TPPM12}. In the tests we set $\nu =1$.

\begin{test}
In this example, we apply homogeneous boundary conditions on the whole $\partial \Omega$, and we choose 
the load term $\mathbf{f}$ in such a way that the analytical solution is 
\[
\mathbf{u}(x,y) =  \begin{pmatrix}
-\frac{1}{2} \, \cos^2(x) \cos(y) \sin(y) \\
\frac{1}{2} \, \cos^2(y) \cos(x) \sin(x)
\end{pmatrix} \qquad 
p(x,y) = \sin(x) - \sin(y).
\]

We consider the error quantities:
\begin{equation}\label{eq:errqnts}
\delta(\mathbf{u}) := \frac{|\mathbf{u}^I - \mathbf{u}_h|_{1,h}}{|\mathbf{u}^I|_{1,h}} \qquad \text{and} 
\qquad \delta(p) := \frac{\|p^I - p_h\|_{L^2}}{\|p^I\|_{L^2}}
\end{equation}
where $|\cdot|_{1,h}$ denotes the norm induced by the discrete bilinear form $a_h(\cdot, \cdot)$.

We compare two different methods, by studying $\delta(\mathbf{u})$ and $\delta(p)$ versus the total number of degrees of freedom $N_{dof}$. 
The first method is the reduced scheme of Section \ref{sec:5} (labeled as ``\text{new}''), with the post-processed pressure of Remark \ref{oss:postproc}.
The second method is the scheme of ~\cite{VEM-elasticity} extended to the Stokes problem (see Remarks \ref{oss:3.1} and \ref{oss:inf-supclassic}), labeled as ``classic''. In both cases we consider polynomial degrees $k=2,3$.

In Figure \ref{Figure2} and \ref{Figure3}, we display the results for the sequence of Voronoi    
meshes $\mathcal{V}_h$. In Figure \ref{Figure4} and \ref{Figure5}, we show the results for the sequence of meshes $\mathcal{T}_h$, while 
in Figure \ref{Figure6} and \ref{Figure7} we plot the results for the sequence of 
meshes $\mathcal{Q}_h$.
 
\begin{figure}[!h]
\centering
\includegraphics[scale=0.3]{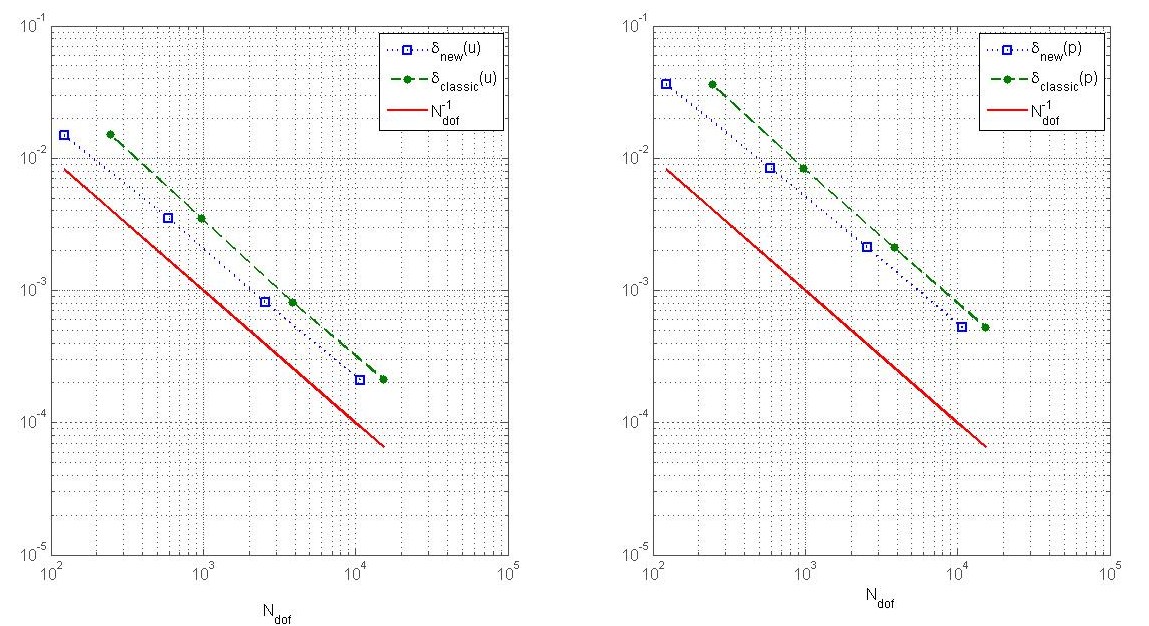}
\caption{Behaviour of $\delta(\mathbf{u}) $ and $\delta(p)$ for the sequence of meshes $\mathcal{V}_h$ with $k=2$.}
\label{Figure2}
\end{figure}

\begin{figure}[!h]
\centering
\includegraphics[scale=0.3]{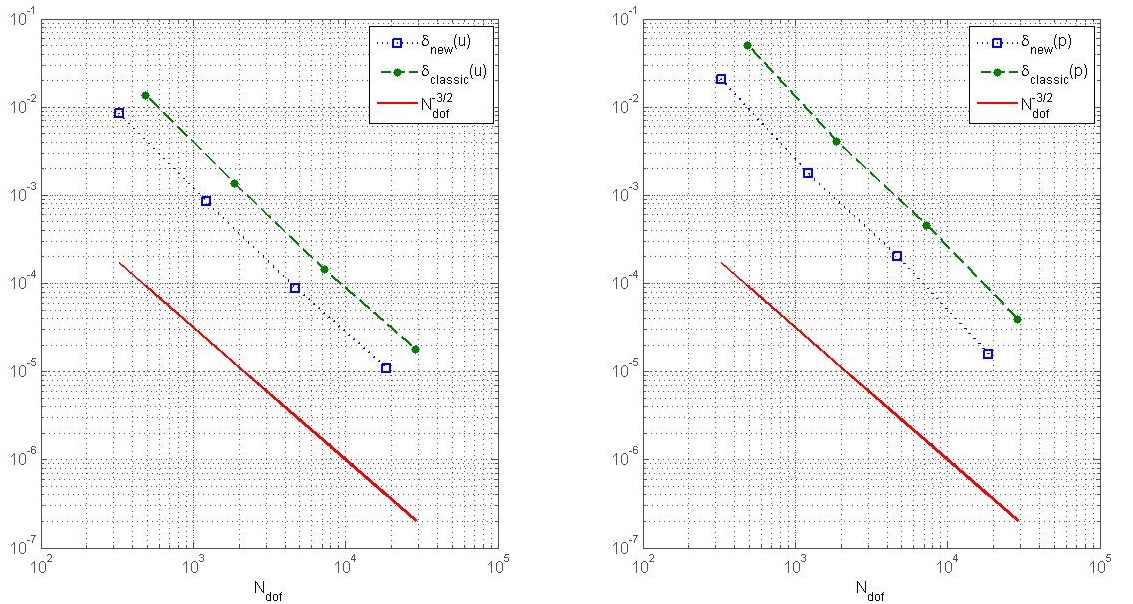}
\caption{Behaviour of $\delta(\mathbf{u}) $ and $\delta(p)$ for the sequence of meshes $\mathcal{V}_h$ with $k=3$.}
\label{Figure3}
\end{figure}

\begin{figure}[!h]
\centering
\includegraphics[scale=0.3]{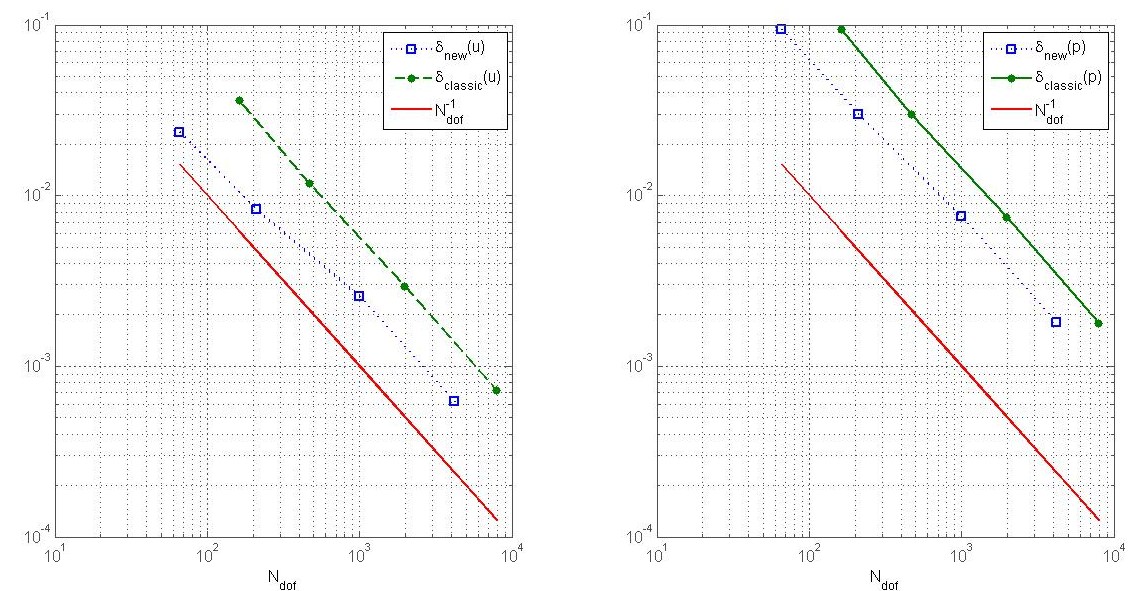}
\caption{Behaviour of $\delta(\mathbf{u}) $ and $\delta(p)$ for the sequence of meshes $\mathcal{T}_h$ with $k=2$.}
\label{Figure4}
\end{figure}

\begin{figure}[!h]
\centering
\includegraphics[scale=0.3]{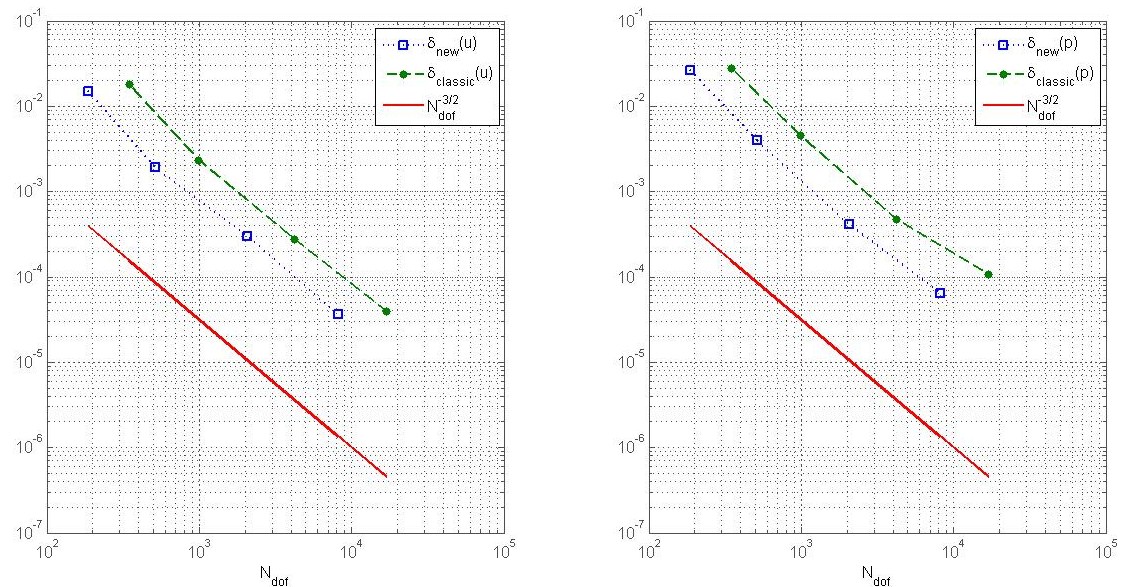}
\caption{Behaviour of $\delta(\mathbf{u}) $ and $\delta(p)$ for the sequence of meshes $\mathcal{T}_h$ with $k=3$.}
\label{Figure5}
\end{figure}

\begin{figure}[!h]
\centering
\includegraphics[scale=0.3]{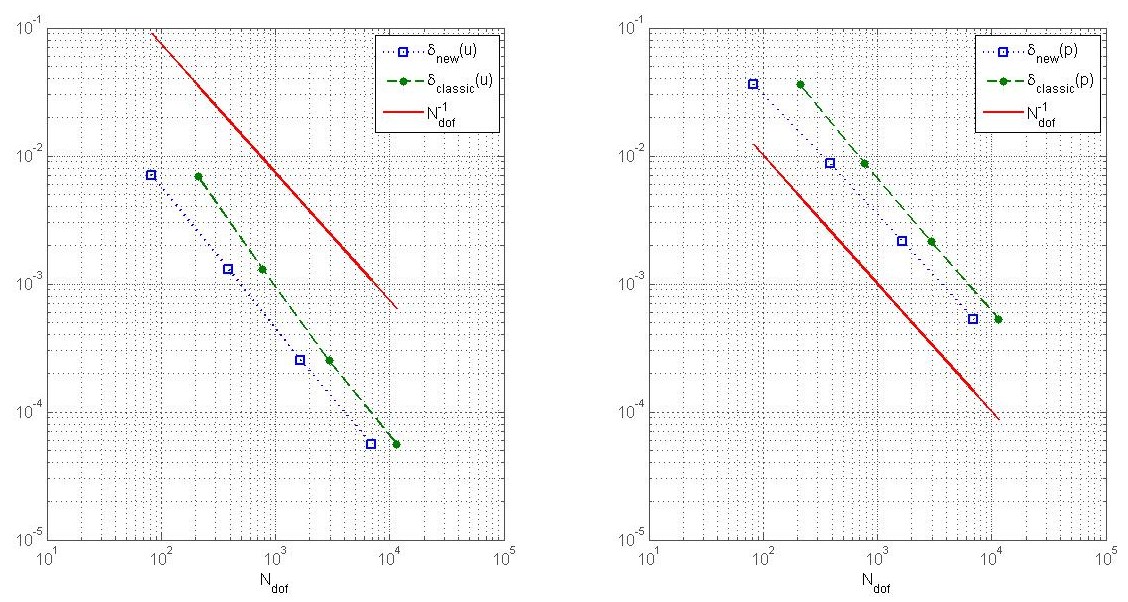}
\caption{Behaviour of $\delta(\mathbf{u}) $ and $\delta(p)$ for the sequence of meshes $\mathcal{Q}_h$ with $k=2$.}
\label{Figure6}
\end{figure}

\begin{figure}[!h]
\centering
\includegraphics[scale=0.3]{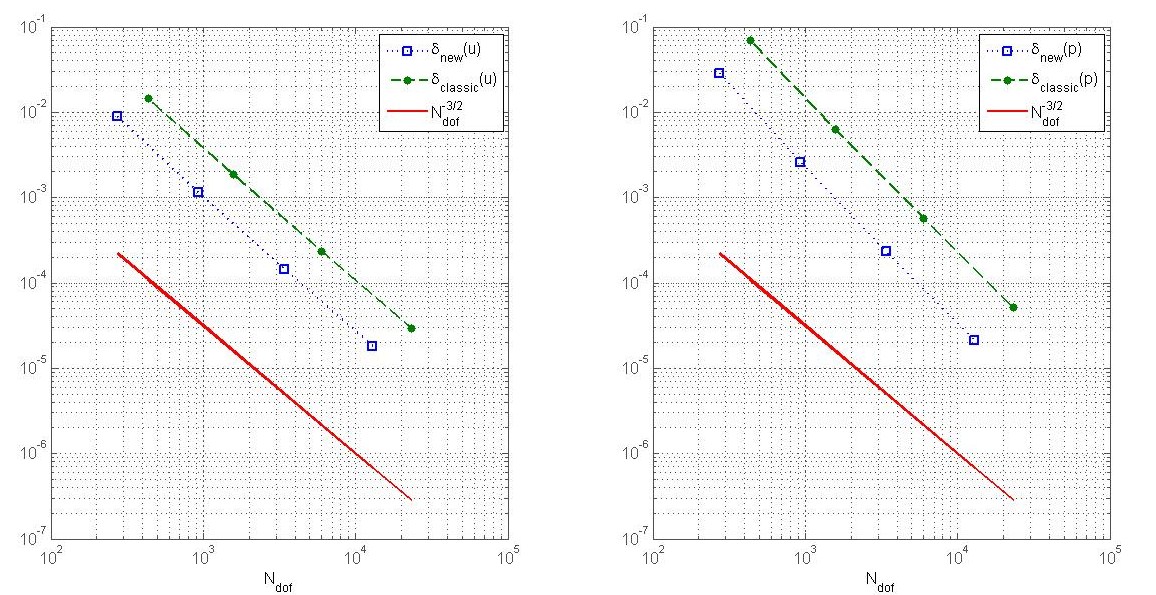}
\caption{Behaviour of $\delta(\mathbf{u}) $ and $\delta(p)$ for the sequence of meshes $\mathcal{Q}_h$ with $k=3$.}
\label{Figure7}
\end{figure}

We notice that the theoretical predictions of Sections \ref{sec:4} and \ref{sec:5} are confirmed (noticed that the method error and $N_{dof}$ behave 
like $h^k$ and $h^{-2}$, respectively).
Moreover, we observe that the reduced method exhibit significant smaller errors than the standard method, at least for this example and with 
the adopted meshes.

\end{test}

\begin{test}
In this example we choose the load term $\mathbf{f}$ and the 
{\em non-homogeneous} polynomial Dirichlet boundary conditions in such a way that the analytical solution is  
\[
\mathbf{u}(x,y) =  \begin{pmatrix}
y^4 + 1\\
x^4 + 2
\end{pmatrix} \qquad 
p(x,y) = x^3 - y^3.
\]

Aim of this test is to check numerically the results of Theorem \ref{thm6}; in order to be more general, we consider the case of non-homogeneous boundary conditions. 
Let $(\mathbf{u}_h, p_h)$ be the solution of Problem \eqref{eq:stokes virtual} and $(\mathbf{\widehat u}_h, \widehat{p}_h)$ be the solution 
of Problem \eqref{eq:stokes reduced virtual}. As a measure of discrepancy between the two solutions, we introduce the parameters
\[
\epsilon(\mathbf{u}) :=  |\widehat I_h \mathbf{u}_h - \mathbf{\widehat u}_h|_{1,h} \qquad  \epsilon(p) := \left(
\sum_{K \in \mathcal{T}_h} \left \| \Pi_0^{0,K} p_h - \widehat{p}_h \right \|_{0,K}^2 \right)^{1/2} ,
\]
where $\widehat I_h \mathbf{u}_h$ denotes the interpolant of $ \mathbf{u}_h$ with respect to the reduced space $\mathbf{\widehat V}_h $,
according to the precedure detailed in \eqref{eq:intdef} and subsequent discussion. 


%
In Table \ref{tabletest} we display the values of $\epsilon(\mathbf{u})$ and $\epsilon(p)$ for the family of meshes $\mathcal{V}_h$, $\mathcal{T}_h$ and $\mathcal{Q}_h$, choosing $k=2,3$. The values of $\epsilon(\mathbf{u})$ and $\epsilon(p)$ confirm the equivalence results provided by Proposition \ref{thm6}. 

\begin{table}[!h]
\centering
\begin{tabular}{ll*{4}{c}}
\toprule
& &\multicolumn{2}{c}{$k=2$} & \multicolumn{2}{c}{$k=3$}\\
\midrule
&           & $\epsilon(\mathbf{u})$   & $\epsilon(p)$      & $\epsilon(\mathbf{u})$  & $\epsilon(p)$        \\
\midrule
\multirow{4}*{$\mathcal{V}_h$}
& $h = 1/4$   &$1.0924681e-13$ &$1.2397027e-13$ &  $2.7665347e-11$ &  $9.5750683e-13$ \\
& $h = 1/8$   &$3.3325783e-13$  & $1.7037760e-13$ &  $2.8147458e-11$ &  $6.4888535e-13$ \\
& $h = 1/16$ &$8.7031014e-13$ &$5.3823612e-13$ &  $3.1718526e-11$ &  $1.5761308e-12$  \\
& $h = 1/32$ &$1.9942180e-12$ &$5.2896229e-13$ &  $7.1270772e-11$ &  $9.7059278e-12$ \\
\midrule
\multirow{4}*{$\mathcal{T}_h$}
& $h = 1/2$   &$1.4647227e-13$ & $2.7158830e-14$ &  $1.4152187e-11$ &  $9.5716694e-13$ \\
& $h = 1/4$   &$4.0200859e-13$ & $6.9691214e-14$ &  $2.3329780e-11$ &  $4.8537564e-13$ \\
& $h = 1/8$   &$1.2058309e-12$ & $1.0109968e-13$ &  $9.7206675e-11$ &  $8.0028046e-12$  \\
& $h = 1/16$ &$2.9427897e-12$ & $2.3636051e-13$ &  $1.9696837e-10$ &  $1.4188633e-11$ \\
\midrule
\multirow{4}*{$\mathcal{Q}_h$}
& $h = 1/4$   &$9.5009907e-14$ & $8.0270859e-14$ &  $1.0298908e-11$ &  $2.1761282e-13$ \\
& $h = 1/8$   &$2.9704999e-13$ & $1.7217954e-13$ &  $4.2678966e-11$ &  $1.5735525e-13$ \\
& $h = 1/16$ &$7.0313002e-13$ & $2.2290502e-13$ &  $2.2776003e-11$  &  $7.9220732e-13$  \\
& $h = 1/32$ &$1.7113467e-12$ & $2.4074145e-13$ &  $6.7792690e-11$ &  $5.9492426e-13$ \\
\bottomrule
\end{tabular}
\caption{$\epsilon(\mathbf{u})$ and $\epsilon(p)$ for the meshes $\mathcal{V}_h$, $\mathcal{T}_h$, $\mathcal{Q}_h$ with $k=2,3$.}
\label{tabletest}
\end{table}

\end{test}

\section{Acknowledgements}
Giuseppe Vacca thanks the National Group of Scientific Computing (GNCS-INDAM) that supported this research through the project:  ``Finanziamento Giovani Ricercatori 2014-2015''.

\addcontentsline{toc}{section}{\refname}
\bibliographystyle{plain}
\bibliography{biblio}

\end{document}